\documentclass[a4paper,11pt]{article}
\usepackage[latin1]{inputenc}
\usepackage[english]{babel}
\usepackage{amsmath}
\usepackage{amsfonts}
\usepackage{amssymb}
\usepackage{epsfig}
\usepackage{amsopn}
\usepackage{amsthm}
\usepackage{color}
\usepackage{graphicx}
\usepackage{subfigure}
\usepackage{enumerate}
\newtheorem{theorem}{Theorem}[section]

\newtheorem{lemma}[theorem]{Lemma}

\newtheorem{definition}[theorem]{Definition}

\newtheorem*{theorem*}{Theorem}
\newtheorem*{lemma*}{Lemma}
\newtheorem*{remark*}{Remark}
\newtheorem*{definition*}{Definition}
\newtheorem*{proposition*}{Proposition}
\newtheorem*{corollary*}{Corollary}
\numberwithin{equation}{section}
\newcommand{\real}{\mathbb{R}}

\let\ced=\c         

\def\L{\Lambda}

\def\qed{\,\unskip\kern 6pt \penalty 500
\raise -2pt\hbox{\vrule \vbox to8pt{\hrule width 6pt
\vfill\hrule}\vrule}\par}
\definecolor{darkblue}{rgb}{0.05, .05, .65}
\definecolor{darkgreen}{rgb}{0.1, .65, .1}
\definecolor{darkred}{rgb}{0.8,0,0}
\newcommand{\beqn}{\begin{equation}}
\newcommand{\eeqn}{\end{equation}}
\newcommand{\bear}{\begin{eqnarray}}
\newcommand{\eear}{\end{eqnarray}}
\newcommand{\bean}{\begin{eqnarray*}}
\newcommand{\eean}{\end{eqnarray*}}
\begin{document}

\title{\huge \bf Blow-up patterns for a reaction-diffusion equation with weighted reaction in general dimension}

\author{
\Large Razvan Gabriel Iagar\,\footnote{Departamento de Matem\'{a}tica
Aplicada, Ciencia e Ingenieria de los Materiales y Tecnologia
Electr\'onica, Universidad Rey Juan Carlos, M\'{o}stoles,
28933, Madrid, Spain, \textit{e-mail:} razvan.iagar@urjc.es},\\
[4pt] \Large Marta Latorre\,\footnote{Departamento de Matem\'{a}tica
Aplicada, Ciencia e Ingenieria de los Materiales y Tecnologia
Electr\'onica, Universidad Rey Juan Carlos, M\'{o}stoles,
28933, Madrid, Spain, \textit{e-mail:} marta.latorre@urjc.es},
\\[4pt] \Large Ariel S\'{a}nchez,\footnote{Departamento de Matem\'{a}tica
Aplicada, Ciencia e Ingenieria de los Materiales y Tecnologia
Electr\'onica, Universidad Rey Juan Carlos, M\'{o}stoles,
28933, Madrid, Spain, \textit{e-mail:} ariel.sanchez@urjc.es}\\
[4pt] }
\date{}
\maketitle

\begin{abstract}
We classify all the blow-up solutions in self-similar form to the following reaction-diffusion equation
$$
\partial_tu=\Delta u^m+|x|^{\sigma}u^p,
$$
posed for $(x,t)\in\real^N\times(0,T)$, with $m>1$, $1\leq p<m$ and $-2(p-1)/(m-1)<\sigma<\infty$. We prove that there are several types of self-similar solutions with respect to the local behavior near the origin, and their existence depends on the magnitude of $\sigma$. In particular, these solutions have different blow-up sets and rates: some of them have $x=0$ as a blow-up point, some other only blow up at (space) infinity. We thus emphasize on the effect of the weight on the specific form of the blow-up patterns of the equation. The present study generalizes previous works by the authors limited to dimension $N=1$ and $\sigma>0$.
\end{abstract}

\


\smallskip


\section{Introduction}

The goal of this paper is to complete the picture of the solutions in self-similar form to the following quasilinear reaction-diffusion equation
\begin{equation}\label{eq1}
u_t=\Delta u^m+|x|^{\sigma}u^p,
\end{equation}
posed for $(x,t)\in\real^N\times(0,T)$ for some $T>0$ and with the following conditions on the exponents $m$, $p$ and $\sigma$
\begin{equation}\label{range.exp}
m>1, \qquad 1\leq p<m, \qquad -\frac{2(p-1)}{m-1}<\sigma<\infty,
\end{equation}
which in particular includes the whole range $\sigma\geq0$ but also some values of $\sigma<0$. With respect to the dimension $N$, our main focus will be on $N\geq2$, since the case $N=1$ has been considered for $\sigma>0$ in a previous work \cite{IS21a}. However, we will also study the case of dimension $N=1$ linked to negative values of $\sigma$, restricted in this case to the range $-1<\sigma<0$ in order to satisfy the usual condition $N+\sigma>0$. In the range of exponents \eqref{range.exp}, we consider self-similar solutions which present finite time blow-up, meaning that there exists a finite time $T\in(0,\infty)$ such that the solution $u(t)\in L^{\infty}(\real^N)$ for any $t\in(0,T)$, but $u$ becomes unbounded at time $T$. This time $T$ is called the blow-up time of $u$. Here and in the sequel we use the standard notation $u(t)$ for the application $x\mapsto u(x,t)$ for a fixed $t>0$. Our interest will be focused on those properties of the solutions related to the blow-up and depending strongly on the weight and the magnitude of $\sigma$.

Reaction-diffusion equations such as Eq. \eqref{eq1} but without a weight at the reaction term are rather well studied by now, a lot of information being available in monographs such as \cite{QS} for the semilinear case $m=1$ and $p>1$ and \cite[Chapter 4]{S4} for $m>1$ and $p>1$ including all the relative positions of $m$ and $p$. In particular, it is shown that for $\sigma=0$ all the solutions to Eq. \eqref{eq1} blow up in finite time for $1<p<m$ but not also for $p=1$, which is a trivial limit case where solutions present exponential grow-up as $t\to\infty$ but are all global. A deeper study of self-similar solutions to Eq. \eqref{eq1} with $\sigma=0$ is performed in \cite[Chapter 4]{S4}, where existence and some properties of such solutions are established, while some further functional properties in higher space dimensions are given in \cite{CdPE98, CdPE02}, see also references therein.

Considering a weight on the reaction term came as a natural step forward once the theory in the spatially homogeneous case was developing with a fast pace. Some nowadays classical papers concerning the semilinear case $m=1$ of Eq. \eqref{eq1} jointly with weighted reaction terms are, for example, \cite{M78} concerning the $\omega$-limit sets of solutions and \cite{BK87, BL89, Pi97, Pi98} where finite time blow-up is considered and the life-span of solutions is studied in dependence on their initial condition and the weight. Andreucci and DiBenedetto study in the long but deep paper \cite{AdB91} the local well-posedness, initial traces and regularity of local solutions to equations such as \eqref{eq1} but with weight $(1+|x|)^{\sigma}$ for a very general range of exponents $m\geq1$, $p>1$ and $\sigma\in\real$, thus including both bounded and unbounded weights. Results related to properties of the exponents $m$, $p$, $\sigma$ and of the initial condition $u_0(x)$ for the solutions to blow up in finite time or to exist globally in time are given for Eq. \eqref{eq1} with $m>1$ in \cite{Qi98, Su02}, the former establishing the Fujita-type exponent, while the existence of global (in time) solutions for $p>m$ sufficiently large is obtained in form of forward self-similar solutions. Moreover, \cite{Su02} also analyzes the interval of existence of solutions to the Cauchy problem with respect to the decay rate of $u_0(x)$ as $|x|\to\infty$, proving that for $p$ sufficiently large, ``fat tails" lead to finite time blow-up, while rapidly decaying initial conditions give solutions that exist globally in time. We also quote in this paragraph the paper \cite{AT05} where a blow-up rate of solutions is given, and works such as \cite{FdPV06, KWZ11, BZZ11, FdP18} where the case of localized, compactly supported weights is considered and quite complete results, including blow-up sets, rates and asymptotic convergence are obtained.

The presence of a spatially dependent coefficient in Eq. \eqref{eq1} produces a second competition in the mechanism of the equation, apart from the already existing one between the diffusion and the reaction effects: it is a competition between the dynamics of the equation near $x=0$ (where formally the weight is very small, tending to zero) and the dynamics of it for large values of $|x|$, where the weight becomes very strong. Thus, a very interesting question related to these models of equations such as Eq. \eqref{eq1} is related to the blow-up set, which can be influenced strongly by the strength of the weighted part. For the semilinear case $m=1$ and in bounded domains, this question has been considered in a series of papers \cite{GLS, GS11, GLS13, GS18} where some conditions for the blow-up to occur at $x=0$ (and more general at the zeros of a generic weight $V(x)$) are given. As a by-product of our analysis, we also give an answer to this question for Eq. \eqref{eq1} by showing that both \emph{blowing up or not at $x=0$ is possible and it depends strongly on the magnitude of $\sigma$}.

Another natural problem is to consider a weight which is unbounded at $x=0$ instead of as $|x|\to\infty$. This is usually called a \emph{singular potential} and its interest stems, in mathematics, from a number of functional inequalities that lead to properties of the solutions, such as, for example, the Hardy inequality. This has been exploited by Baras and Goldstein in their classical work \cite{BG84}, generalized later in \cite{CM99}, to show that letting $m=p=1$ and $\sigma=-2$ produces a very interesting and unexpected balance between existence and non-existence (explained as an \emph{instantaneous complete blow-up}) of the solutions. Later on, singular potentials were considered in a number of different equations and models, such as for example fast diffusion $0<m<1$ in \cite{GK03, GGK05} and doubly nonlinear diffusion in \cite{Ko04}. Recently, the functional analysis of solutions to the semilinear problem with singular potential $|x|^{\sigma}$ with $-2<\sigma<0$ became fashionable, with a number of works studying this problem, see for example \cite{BSTW17, BS19, CIT21a, CIT21b, T20, HT21}. Similarity solutions and behavior near blow-up for general solutions were addressed by Filippas and Tertikas \cite{FT00} in the semilinear case $m=1$ with $\sigma>-2$. Their analysis has been completed very recently by Mukai and Seki \cite{MS21} with a study of the so-called \emph{Type II blow-up} for $p$ sufficiently large.

The present paper extends, on the one hand, to dimension $N\geq2$, in the range $\sigma\geq0$, previous results by two of the authors \cite{IS19, IS21a} concerning the classification of self-similar solutions to Eq. \eqref{eq1}. This generalization, as we shall see, introduces rather tough technical challenges, since some techniques used in dimension $N=1$ do no longer apply for $N\geq2$ and some previous results for $\sigma=0$ were no longer available in higher space dimension. This led us to consider new and rather difficult geometric barriers for the flow in the dynamical system, as seen in Sections \ref{sec.small} and \ref{sec.large} below. On the other hand, we have noticed that a natural limit for $\sigma$ for many technical steps to hold true is not $\sigma=0$ but $\sigma=-2$, and thus we extend our analysis and classification to the whole range $\sigma>-2$. However, some differences in the form and local behavior of a part of the profiles appear also when passing through $\sigma=0$. The results obtained here for $\sigma=-2$ also complete the remaining ranges after previous works by two of the authors such as \cite{IMS21}, where the opposite range for $\sigma$, namely $-2<\sigma<-2(p-1)/(m-1)$, is analyzed and it is proved that finite time blow-up should not occur in this range, or the short note \cite{IS20} for $\sigma=-2$, $p=m$ where an interesting phenomenon of blow-up only at $x=0$ but with suitable integrability properties has been evidenced.

It is now the moment to explain with more rigor our main results and techniques.

\medskip

\noindent \textbf{Main results.} As previously explained, we are looking for self-similar solutions to \eqref{eq1} presenting finite time blow-up, that is, in backward form
\begin{equation}\label{SSS}
u(x,t)=(T-t)^{-\alpha}f(\xi), \qquad \xi=|x|(T-t)^{\beta},
\end{equation}
where the exponents $\alpha$, $\beta$ and the profiles $f$ are to be determined. Inserting the ansatz \eqref{SSS} into Eq. \eqref{eq1}, we readily obtain that the self-similarity exponents are both explicit and positive
\begin{equation}\label{SS.exp}
\alpha=\frac{\sigma+2}{\sigma(m-1)+2(p-1)}, \qquad \beta=\frac{m-p}{\sigma(m-1)+2(p-1)},
\end{equation}
while the profiles $f(\xi)$ solve the following differential equation
\begin{equation}\label{SSODE}
(f^m)''(\xi)+\frac{N-1}{\xi}(f^m)'(\xi)-\alpha f(\xi)+\beta \xi f'(\xi)+\xi^{\sigma}f(\xi)^p=0, \qquad \xi\in[0,\infty).
\end{equation}
This differential equation will be our main object of study in the present work. Similarly as in the one-dimensional case \cite{IS21a}, we introduce our concept of profile we are looking for in the next
\begin{definition}\label{def1}
We say that a solution $f$ to \eqref{SSODE} is a \textbf{good profile}
if it fulfills one of the following two properties related to its
behavior at $\xi=0$:

\indent (H1) $f(0)=a>0$ if $N\geq2$ or $f(0)=a>0$ and $f'(0)=0$ if $N=1$.

\indent (H2) $f(0)=0$, $(f^m)'(0)=0$.

A good profile $f$ is called a \textbf{good profile with interface}
at some point $\eta\in(0,\infty)$ if
$$
f(\eta)=0, \qquad (f^m)'(\eta)=0, \qquad f>0 \ {\rm on} \
(\eta-\delta,\eta), \ {\rm for \ some \ } \delta>0.
$$
\end{definition}
Our first result is an existence theorem in which we do not make any difference between the local behavior of profiles.
\begin{theorem}\label{th.exist}
Let $m$, $p$ and $\sigma$ be as in \eqref{range.exp}. Then there exists at least one good profile with interface to Eq. \eqref{eq1}, according to Definition \ref{def1}.
\end{theorem}
However, the most interesting part of the analysis is, in our opinion, the classification of the good self-similar profiles with interface (and thus of the self-similar solutions to Eq. \eqref{eq1} according to \eqref{SSS}) with respect to their local behavior as $\xi\to0$. An outcome of the analysis in this work will give that there are exactly three possible local behaviors of the good profiles according to their behaviors at $\xi=0$, namely
\begin{itemize}
  \item Profiles with positive value at $\xi=0$, that is $f(0)>0$ and
\begin{equation}\label{beh.Q1}
f(\xi)\sim\left\{\begin{array}{ll}\left[K+\frac{\alpha(m-1)}{2mN}\xi^{2}\right]^{1/(m-1)}, & {\rm if} \ \sigma>0,\\
\left[K+\frac{(m-1)(\alpha-K^{(p-1)/(m-1)})}{2mN}\xi^2\right]^{1/(m-1)}, & {\rm if} \ \sigma=0,\\
\left[K-\frac{m-p}{m(N+\sigma)(\sigma+2)}\xi^{\sigma+2}\right]^{1/(m-p)}, & {\rm if} \ \sigma<0\end{array}\right.
\end{equation}
as $\xi\to0$, where $K>0$ is an arbitrary constant.
  \item Profiles with $f(0)=0$ and local behavior given in a first order approximation by
\begin{equation}\label{beh.P2}
f(\xi)\sim\left[\frac{m-1}{2m(mN-N+2)}\right]^{1/(m-1)}\xi^{2/(m-1)}
\end{equation}
as $\xi\to0$.
  \item Profiles with $f(0)=0$ and local behavior given in a first order approximation by
\begin{equation}\label{beh.P0}
f(\xi)\sim K\xi^{(\sigma+2)/(m-p)}
\end{equation}
as $\xi\to0$, where $K>0$ is an arbitrary constant.
\end{itemize}
Let us remark at this point that there is a qualitative difference between the profiles with positive value $f(0)>0$. Indeed, for $\sigma>0$ all these profiles start increasingly in a right neighborhood of $\xi=0$ (a fact similar to the case $N=1$ in \cite{IS21a}), for $\sigma<0$ the profiles start in a decreasing way in a right neighborhood of $\xi=0$, while in the limiting case $\sigma=0$ (also noticed in \cite[Chapter 4]{S4}) there exists a constant solution $f(\xi)=\alpha^{1/(p-1)}=K^*$ and the profiles might start either in an increasing way or in a decreasing way according to whether $f(0)>K^*$ (decreasing) or $f(0)<K^*$ (increasing). 

It is immediate to check that profiles with local behavior \eqref{beh.P2} and \eqref{beh.P0} satisfy assumption (H2) in Definition \ref{def1}, provided \eqref{range.exp} holds true. The natural question that arises in view of Theorem \ref{th.exist} and the previous list of local behaviors is to classify the good profiles with interface with respect to their local behavior at the origin. This is where the magnitude of $\sigma$ comes into play. More precisely, we have the following classification:
\begin{theorem}\label{th.class}
Let $m$, $p$ and $\sigma$ be as in \eqref{range.exp}. We then have
\begin{enumerate}
\item There exists $\sigma_0>0$ such that if
$$
-\frac{2(p-1)}{m-1}<\sigma<\sigma_0,
$$
all the good profiles with interface to Eq. \eqref{SSODE} present the local behavior \eqref{beh.Q1} as $\xi\to0$.
\item There exists $\sigma_1>\sigma_0$ sufficiently large such that if $\sigma\in(\sigma_1,+\infty)$, all the good profiles with interface to Eq. \eqref{SSODE} present the local behavior \eqref{beh.P0} as $\xi\to0$.
\item There exists at least one $\sigma^*\in[\sigma_0,\sigma_1]$ such that, for $\sigma=\sigma^*$, there exists a good profile with interface to Eq. \eqref{SSODE} presenting the local behavior \eqref{beh.P2} as $\xi\to0$.
\end{enumerate}
\end{theorem}
We leave here as a \textbf{conjecture} that in fact $\sigma_0=\sigma_1=\sigma^*$, hence the uniqueness of the value of $\sigma$ for which the behavior \eqref{beh.P2} is taken. Both intuition and numerical evidence suggest that this is true, and there is thus a continuous process of change of the geometric form of the profiles as $\sigma$ increases. However, in order to prove this conjecture rigorously, some monotonicity properties of the dynamics of the equation \eqref{SSODE} with respect to $\sigma$ are needed, and proving such monotonicity is usually a very difficult problem when one deals with self-similar solutions in backward form. We leave below some comments and remarks related to the classification.

\medskip

\noindent \textbf{Blow-up sets.} A very important influence of the classification given in Theorem \ref{th.class} appears in relation to the blow-up sets of the self-similar solutions \eqref{SSS} with good profiles $f(\xi)$ as in Theorem \ref{th.class}. We recall below the definition of the blow-up set adapted from \cite[Section 24]{QS}: for any solution $u$ to Eq. \eqref{eq1} with finite blow-up time $T\in(0,\infty)$, the \emph{blow-up set} of $u$ is defined as
\begin{equation}\label{BUS}
B(u)=\{x\in\real^N: \exists(x_k,t_k)\in\real^N\times(0,T), \ t_k\to T, \ x_k\to x, \ {\rm and} \  |u(x_k,t_k)|\to\infty, \ {\rm as} \ k\to\infty\}.
\end{equation}
On the one hand, we readily notice that for either a good profile with local behavior \eqref{beh.Q1}, or for a good profile with local behavior \eqref{beh.P2} as $\xi\to0$, the blow-up set is the whole space $\real^N$. Indeed, we infer from \eqref{SSS} that either
\begin{equation}\label{interm1}
u(x,t)=(T-t)^{-\alpha}f(|x|(T-t)^{\beta})\sim a(T-t)^{-\alpha}, \quad {\rm as} \ t\to T,
\end{equation}
for self-similar solutions with good profiles presenting the local behavior at the origin given by \eqref{beh.Q1} with $a=K^{1/(m-1)}$ if $\sigma\geq0$ or $a=K^{1/(m-p)}$ if $\sigma<0$, or
\begin{equation}\label{interm2}
u(x,t)\sim C(T-t)^{-\alpha+2\beta/(m-1)}|x|^{2/(m-1)}=C(T-t)^{-1/(m-1)}|x|^{2/(m-1)}, \quad {\rm as} \ t\to T,
\end{equation}
for self-similar solutions with good profiles presenting the local behavior at the origin given by \eqref{beh.P2}. It is obvious from \eqref{interm1}, \eqref{interm2} and the definition \eqref{BUS} that in both cases the blow-up set is the whole space, but the \emph{blow-up rate over fixed compact sets is different} in the two cases. On the other hand, for self-similar solutions whose profiles behave locally as in \eqref{beh.P0} as $\xi\to0$, a sharp difference appears: indeed, for any $x\in\real^N$ fixed, we find
$$
u(x,t)\sim C(T-t)^{-\alpha+(\sigma+2)\beta/(m-p)}|x|^{(\sigma+2)/(m-p)}=C|x|^{(\sigma+2)/(m-p)}<\infty,
$$
hence these solutions remain bounded forever at any finite point. However, finite time blow-up still occurs at $t=T$ but only on curves $x(t)$ depending on $t$ such that $x(t)\to\infty$ as $t\to T$. This phenomenon is known in literature as \emph{blow-up at (space) infinity}, see also \cite{La84, GU05, GU06} for other examples when it occurs in the semilinear case. Let us thus end this discussion by stressing here that in particular \textbf{the origin can be a blow-up point} for a solution to Eq. \eqref{eq1} \textbf{or not}, and this depends on how large is $\sigma$, giving thus a partial answer to the question discussed in this introduction.

\medskip

\noindent \textbf{Remark.} The results in this paper generalize, on the one hand, the ones obtained in \cite{IS19, IS21a} in the one-dimensional case for $\sigma>0$, proving a similar classification for the profiles. For the admitted range of negative values of $\sigma$ according to \eqref{range.exp}, our analysis shows that there is a single type of good profiles with interface, that will be decreasing with respect to $\xi$ and thus completing the outcome of the recent paper \cite{IMS21} with the complementary range of negative $\sigma$, that is, when the important constant
\begin{equation}\label{constantL}
L:=\sigma(m-1)+2(p-1)
\end{equation}
is positive. We show that in this case finite time blow-up may occur when $L>0$, while this is not true for $L<0$ as shown in \cite{IMS21}. Finally, we also provide an alternative and independent proof of the existence of a self-similar solution for the homogeneous case $\sigma=0$, established in \cite[Chapter 4]{S4}.

\medskip

\noindent \textbf{Organization of the paper}. The main tool of this work is a phase-space analysis applied to a quadratic dynamical system of three equations into which \eqref{SSODE} is mapped through a change of variable that is introduced in Section \ref{sec.local}. The critical points in the phase space will be classified in Section \ref{sec.local} for the finite part of the space, respectively Section \ref{sec.infty} for the infinity of the phase space. Once the local analysis has been understood, it is time for the global analysis of the connections in the phase space, which will cover the rest of the work. In a first step, the proof of the existence Theorem \ref{th.exist} is performed in Section \ref{sec.exist}, using a strategy of \emph{backward shooting from the interface point} but which in this case will be technically different from the analogous result in \cite{IS21a}: in the latter, an analysis directly in terms of profiles and using arguments of continuity near $x=0$ was used, while in the present paper, such arguments are no longer valid at least for exponents $\sigma<0$ where continuity with respect to $\sigma$ cannot be used anymore. Thus, all the proofs in this paper \emph{are performed at the level of the quadratic dynamical system}, where the possible singularity at $x=0$ is removed, and this also allows us to have an independent proof also for $\sigma=0$ of the previous results in the book \cite{S4}. The proof of Theorem \ref{th.class} is split into Section \ref{sec.small} and Section \ref{sec.large}, in both being used geometric arguments such as barriers for the flow of the dynamical system in forms of planes or surfaces, but generating very different invariant regions according to whether $\sigma$ is small or large. Notice here that the proofs in these sections strongly depart from their analogous results in dimension $N=1$ and $\sigma>0$, and are far more involved technically. We end this presentation by stressing that the fact that for $\sigma$ large there exist only good profiles with behavior as $\xi\to0$ given by \eqref{beh.P0} is novel also in dimension $N=1$, as we were not able to prove it in the dedicated paper \cite{IS21a} with the tools developed therein.

\section{The phase space. Local analysis of the finite critical points}\label{sec.local}

As explained in the Introduction, we focus on the differential equation \eqref{SSODE} satisfied by the self-similar profiles $f(\xi)$. Since this equation is non-autonomous and not easy to study by direct methods, we employ a phase space analysis. More precisely, we introduce the following change of variables also used in previous works such as \cite{IS19,IMS21}
\begin{equation}\label{PSchange}
X(\eta)=\frac{m}{\alpha}\xi^{-2}f(\xi)^{m-1}, \ \ Y(\eta)=\frac{m}{\alpha}\xi^{-1}f(\xi)^{m-2}f'(\xi), \ \ Z(\eta)=\frac{1}{\alpha}\xi^{\sigma}f(\xi)^{p-1},
\end{equation}
with a new independent variable $\eta$ defined implicitly in terms of $\xi$ as
\begin{equation}\label{ind.var}
\eta(\xi)=\frac{\alpha}{m}\int_{0}^{\xi}\zeta f(\zeta)^{1-m}\,d\zeta.
\end{equation}
By expressing $f'(\xi)$ in terms of $Y$ from the second equation in \eqref{PSchange} and then performing some direct calculations, we find that Eq. \eqref{SSODE} is mapped by \eqref{PSchange} into the following three-dimensional dynamical system:
\begin{equation}\label{PSsyst1}
\left\{\begin{array}{ll}\dot{X}=X[(m-1)Y-2X],\\
\dot{Y}=-Y^2-\frac{\beta}{\alpha}Y+X-NXY-XZ,\\
\dot{Z}=Z[(p-1)Y+\sigma X].\end{array}\right.
\end{equation}
Let us first observe that the planes $\{X=0\}$ and $\{Z=0\}$ are invariant for the system \eqref{PSsyst1} and also infer from \eqref{PSchange} that $X\geq0$, $Z\geq0$. Assume now that $p>1$ (the limit case $p=1$ being very similar and treated at the end of the current section). The critical points in the finite part of the phase space associated to the system \eqref{PSsyst1} are then
$$
P_0=(0,0,0), \ \ P_1=\left(0,-\frac{\beta}{\alpha},0\right), \ \ P_2=(X(P_2),Y(P_2),0) \ \ {\rm and} \ \ P_{\gamma}=(0,0,\gamma),
$$
for any $\gamma>0$, where
\begin{equation}\label{comp.P2}
X(P_2)=\frac{m-1}{2\alpha(mN-N+2)}, \qquad Y(P_2)=\frac{1}{\alpha(mN-N+2)}.
\end{equation}
Let us remark at this point that $Y(P_2)<1/N$, since
\begin{equation}\label{ineq1}
\frac{1}{N}-Y(P_2)=\frac{2(N(m-p)+\sigma+2)}{N(\sigma+2)(mN-N+2)}>0.
\end{equation}
We next analyze the flow of the system \eqref{PSsyst1} near these critical points.
\begin{lemma}[Local analysis near $P_0$]\label{lem.P0}
In a neighborhood of the critical point $P_0$ the system \eqref{PSsyst1} has a one-dimensional stable manifold and a two-dimensional center manifold. The connections tangent to the center manifold go out of $P_0$ and contain profiles with the local behavior \eqref{beh.P0} as $\xi\to0$.
\end{lemma}
\begin{proof}
The linearization of the system \eqref{PSsyst1} near $P_0$ has the matrix
$$
M(P_0)=\left(
         \begin{array}{ccc}
           0 & 0 & 0 \\
           1 & -\frac{\beta}{\alpha} & 0 \\
           0 & 0 & 0 \\
         \end{array}
       \right),
$$
with one negative eigenvalue and a two-dimensional center manifold. We introduce a new change of variable by setting $U=(\beta/\alpha)Y-X$ in order to replace the variable $Y$ and get a canonical form:
\begin{equation}\label{interm3}
\left\{\begin{array}{ll}\dot{X}=\frac{1}{\beta}X^2+\frac{(m-1)\alpha}{\beta}XU,\\
\dot{U}=-\frac{\beta}{\alpha}U-\frac{\alpha}{\beta}U^2-\frac{\alpha(m+1)+N\beta}{\beta}XU-\frac{\beta}{\alpha}XZ-\frac{(N-2)\beta+m\alpha}{\beta}X^2,\\
\dot{Z}=\frac{1}{\beta}XZ+\frac{\alpha(p-1)}{\beta}ZU.\end{array}\right.
\end{equation}
According to \cite[Theorem 3, Section 2.5]{Carr}, we look for a center manifold in form of a surface whose first order Taylor approximation has the form
$$
U=h(X,Z)=aX^2+bXZ+cZ^2+O(|(X,Z)|^3),
$$
with $a$, $b$ and $c$ coefficients to be determined later. By replacing this expression into the equation of the center manifold given in \cite[Theorem 1, Section 2.12]{Pe} and identifying the similar quadratic terms, we find the expansion of the center manifold near $P_0$ as
\begin{equation}\label{center.P0}
h(X,Z)=-\frac{\alpha}{\beta^2}[(N-2)\beta+\alpha m]X^2-XZ+XO(|(X,Z)|^2),
\end{equation}
where the fact that the higher order terms are all multiples of $X$ follows from an easy argument by induction based on the non-appearance of pure powers of $Z$ alone in the vector field of the system \eqref{PSsyst1}. We omit these straightforward but rather long calculations, more details are given in \cite[Section 2]{IS21a}. By replacing next $U$ by $h(X,Z)$ in the first and third equation of the system \eqref{interm3}, we infer from the same theoretical result \cite[Theorem 1, Section 2.12]{Pe} that the flow on the center manifold is given by the reduced system
\begin{equation}\label{interm4}
\left\{\begin{array}{ll}\dot{X}&=\frac{1}{\beta}X^2+X^2O(|(X,Z)|),\\
\dot{Z}&=\frac{1}{\beta}XZ+XO(|(X,Z)|^2),\end{array}\right.
\end{equation}
in a neighborhood of its origin $(X,Z)=(0,0)$. It thus follows that the orbits on the center manifold go out of the point $P_0$. The local behavior of the profiles contained in these orbits is deduced by direct integration of the system \eqref{interm4}, leading to $Z\sim KX$ in a first approximation for any free constant $K>0$ and then to \eqref{beh.P0} by replacing $Z$ and $X$ with their definitions in \eqref{PSchange}. Since $X\to0$ at $P_0$, we get that
$$
X(\xi)=\frac{m}{\alpha}\xi^{-2}f(\xi)^{m-1}\sim\xi^{-2+(m-1)(\sigma+2)/(m-p)}=\xi^{[\sigma(m-1)+2(p-1)]/(m-p)}\to0,
$$
thus the behavior \eqref{beh.P0} is taken as $\xi\to0$ owing to the fact that $\sigma(m-1)+2(p-1)>0$ in our range \eqref{range.exp}. Finally, the one-dimensional stable manifold of $P_0$ is contained in the $Y$ axis, as indicated by the eigenvector $e_2=(0,1,0)$ corresponding to the eigenvalue $-\beta/\alpha$ of the matrix $M(P_0)$ and by the uniqueness of the stable manifold \cite[Theorem 3.2.1]{GH}, and no profiles are contained in it.
\end{proof}
It is now the turn for the point $P_1$, which will codify the interface behavior.
\begin{lemma}[Local analysis near $P_1$]\label{lem.P1}
The system \eqref{PSsyst1} in a neighborhood of $P_1$ has a two-dimensional stable manifold and a one-dimensional unstable manifold. The orbits entering $P_1$ on the two-dimensional manifold correspond to an interface at a point $\xi_0\in(0,\infty)$, with the more precise local behavior
\begin{equation}\label{beh.P1}
f(\xi)\sim\left[C-\frac{\beta(m-1)}{2m}\xi^2\right]_{+}^{1/(m-1)}, \qquad {\rm as} \ \xi\to\xi_0=\sqrt{\frac{2mC}{\beta(m-1)}}, \ \xi<\xi_0,
\end{equation}
where $C>0$ is a free constant.
\end{lemma}
\begin{proof}
The linearization of the system \eqref{PSsyst1} in a neighborhood of $P_1$ has the matrix
$$
M(P_1)=\left(
         \begin{array}{ccc}
           -\frac{(m-1)\beta}{\alpha} & 0 & 0 \\[1mm]
           1+\frac{N\beta}{\alpha} & \frac{\beta}{\alpha} & 0 \\[1mm]
           0 & 0 & -\frac{(p-1)\beta}{\alpha} \\
         \end{array}
       \right)
$$
thus the two-dimensional stable manifold and the one-dimensional unstable manifold are obvious. The local behavior near $P_1$ is given by the fact that $Y\to-\beta/\alpha$ on the orbits entering $P_1$, together with the fact that $X\to0$ and $Z\to0$. If this behavior would be taken as $\xi\to\infty$, since $Y(\xi)=m\xi^{-1}(f(\xi)^{m-1})'(\xi)/(m-1)$, the fact that
$$
\xi X'(\xi)=-2X(\xi)+(m-1)Y(\xi)
$$
together with an application of \cite[Lemma 2.9]{IL13} for the function $X(\xi)$ would imply that there exists a sequence $\xi_k\to\infty$ such that $(m-1)Y(\xi_k)\to0$, which is a contradiction. We thus deduce that the local behavior is taken, in terms of profiles, as $\xi\to\xi_0\in(0,\infty)$ from the left, which gives first that $f(\xi_0)=0$ and then
$$
(f^{m-1})'(\xi)\sim\frac{\beta(m-1)\xi}{m}, \qquad {\rm as} \ \xi\to\xi_0,
$$
whence the local behavior given by \eqref{beh.P1} follows by integration on a generic interval $(\xi,\xi_0)$.
\end{proof}
We now turn out our attention to the next critical point $P_2$.
\begin{lemma}[Local analysis near $P_2$]\label{lem.P2}
The system \eqref{PSsyst1} in a neighborhood of the critical point $P_2$ has a two-dimensional stable manifold and a one-dimensional unstable manifold. The stable manifold is included in the plane $\{Z=0\}$. There exists a unique orbit going out of $P_2$, which contains profiles with local behavior given by \eqref{beh.P2} as $\xi\to0$.
\end{lemma}
\begin{proof}
Setting
$$
\varphi:=\alpha(mN-N+2),
$$
we compute the matrix of the linearization of the system \eqref{PSsyst1} near $P_2$ as follows
$$
M(P_2)=\frac{1}{\varphi}\left(
  \begin{array}{ccc}
    -(m-1) & \frac{(m-1)^2}{2} & 0 \\[1mm]
    \varphi-N & -\frac{\beta\varphi}{\alpha}-\frac{mN-N+4}{2} & -\frac{m-1}{2} \\[1mm]
    0 & 0 & \sigma(m-1)+2(p-1) \\
  \end{array}
\right)
$$
with eigenvalues $\lambda_1$, $\lambda_2$ and
$$
\lambda_3=\frac{2(p-1)+\sigma(m-1)}{\varphi}>0, \qquad {\rm for} \ \sigma>-\frac{2(p-1)}{m-1}.
$$
Since $M(P_2)$ is a block matrix, by standard linear algebra results related to the trace and the determinant we have
$$
\lambda_1+\lambda_2<0, \qquad \lambda_1\lambda_2=\frac{m-1}{2\alpha\varphi}>0,
$$
hence either $\lambda_1<0$ and $\lambda_2<0$ or $\lambda_1$, $\lambda_2$ are conjugated complex numbers with negative real parts. Moreover, the eigenvectors $e_1$ and $e_2$ of $M(P_2)$ corresponding to these eigenvalues have zero $Z$-component, thus the critical point $P_2$ is a local stable node or stable focus inside the invariant plane $\{Z=0\}$. Since $\lambda_3>0$, there exists a unique orbit going out of $P_2$ towards the interior of the phase space, tangent to the eigenvector $e_3$ corresponding to the eigenvalue $\lambda_3$, whose components are
\begin{equation}\label{vector.P2}
\begin{split}
&X(\sigma)=-(m-1)^3<0, \qquad Y(\sigma)=-(m-1)[(\sigma+2)(m-1)+2(p-1)]<0,\\
&Z(\sigma)=(\sigma+2)^2(m-1)^2+(\sigma+2)[(N-2)(m-1)^2+4p(m-1)]\\&+2N(m-1)^2+4(p-1)^2-4(m-2)(p-1)>0.
\end{split}
\end{equation}
We thus deduce that the orbit going out of $P_2$ starts in a small neighborhood of the point decreasingly in $X$ and $Y$ and (of course) increasingly in $Z$. Moreover, this orbit contains profiles such that, in a first approximation, $X(\xi)\sim X(P_2)$, which readily leads to \eqref{beh.P2}. Since
$$
Z(\xi)=\frac{1}{\alpha}\xi^{\sigma}f(\xi)^{p-1}\sim K\xi^{(\sigma(m-1)+2(p-1))/(m-1)}\to0,
$$
we infer that the local behavior \eqref{beh.P2} is taken as $\xi\to0$, ending the proof.
\end{proof}
We are left with the family of critical points $P_{\gamma}$.
\begin{lemma}[Local analysis near $P_{\gamma}$]\label{lem.Pg}
For
\begin{equation}\label{eq.gamma}
\gamma=\gamma_0:=\frac{1}{\alpha(p-1)},
\end{equation}
the critical point $P_{\gamma_0}$ behaves as an attractor for the orbits approaching it from the half-space $\{X>0\}$. The profiles contained in the orbits entering $P_{\gamma_0}$ have a tail at infinity
\begin{equation}\label{queue}
f(\xi)\sim\left(\frac{1}{p-1}\right)^{1/(p-1)}\xi^{-\sigma/(p-1)}, \qquad {\rm as} \ \xi\to\infty.
\end{equation}
For $\gamma\in(0,\infty)$ with $\gamma\neq\gamma_0$, there is no profile contained in any orbit either entering or going out of $P_{\gamma}$.
\end{lemma}
\begin{proof}
We follow similar steps as in the proof of \cite[Lemma 2.4]{IS21a} or \cite[Lemma 2.3]{IMS21}. We first translate the critical point to the origin by letting $Z=\overline{Z}+\gamma$ and obtaining the new system
\begin{equation}\label{interm5}
\left\{\begin{array}{ll}\dot{X}=X[(m-1)Y-2X],\\
\dot{Y}=-Y^2-\frac{\beta}{\alpha}Y+(1-\gamma)X-NXY-X\overline{Z},\\
\dot{\overline{Z}}=\overline{Z}[(p-1)Y+\sigma X]+\sigma\gamma X+(p-1)\gamma
Y.\end{array}\right.
\end{equation}
The linearization of the system \eqref{interm5} near the origin has a one-dimensional stable manifold and a two-dimensional center manifold. As usual, the most involved part of the analysis is the study of the center manifold. We next perform a double change of variable in \eqref{PSsyst1} in order to reduce some linear terms and to put it into a canonical form, by replacing $Y$ and $\overline{Z}$ with the new variables $W$ and $V$ defined by
$$
W=\frac{\beta}{\alpha}Y-(1-\gamma)X, \qquad V=\overline{Z}+kY, \qquad k=\frac{\alpha(p-1)\gamma}{\beta},
$$
to get the following new system in variables $(X,W,V)$:
\begin{equation}\label{interm6}
\left\{\begin{array}{ll}\dot{X}=\left[\frac{(m-1)\alpha(1-\gamma)}{\beta}-2\right]X^2+\frac{(m-1)\alpha}{\beta}WX,\\
\dot{W}=-\frac{\beta}{\alpha}W-\frac{\alpha}{\beta}W^2-\frac{\beta}{\alpha}VX-D_1X^2-D_2XW,\\
\dot{V}=(\gamma\sigma+k(1-\gamma))X+D_3XV-D_4X^2-\frac{\alpha^2kp}{\beta^2}W^2-D_5XW+\frac{\alpha(p-1)}{\beta}VW,\end{array}\right.
\end{equation}
with coefficients
\begin{equation}\label{interm6b}
\begin{split}
&D_1=\frac{(1-\gamma)[(N-2)\beta+m\alpha(1-\gamma)-\alpha\gamma(p-1)]}{\beta}, \ D_2=\frac{N\beta+\alpha(m+1)-\alpha\gamma(m+p)}{\beta},\\
&D_3=\sigma-k+\frac{(p-1)\alpha(1-\gamma)}{\beta}, \ D_4=\frac{(1-\gamma)\alpha^2\gamma(p-1)[(N+\sigma)\beta+\alpha(\gamma+p-2p\gamma)]}{\beta^3},\\
&D_5=\frac{\alpha^2\gamma(p-1)[(N+\sigma)\beta+\alpha(\gamma+2p-3p\gamma)]}{\beta^3}.
\end{split}
\end{equation}
By applying the center manifold theorem \cite[Theorem 1, Section 2.12]{Pe} to the system \eqref{interm6}, we obtain that the two-dimensional center manifold at the origin in \eqref{interm6} has the form
$$
W=AX^2-XV+XO(|(X,V)|^2), \qquad A=\frac{\alpha}{\beta}\left[\gamma\sigma+k(1-\gamma)-D_1\right],
$$
where $D_1$ is the first coefficient in \eqref{interm6b} and the fact that the higher order terms are a multiple of $X$ follows by induction taking into account the non-appearance of pure powers of $V$ in the vector field of the system \eqref{interm6}. Therefore, the flow on the center manifold is given by the following reduced system
\begin{equation}\label{flowgamma}
\left\{\begin{array}{ll}\dot{X}=\left[\frac{(m-1)\alpha(1-\gamma)}{\beta}-2\right]X^2+X^2O(|(X,V)|),\\
\dot{V}=\left[\sigma\gamma+k(1-\gamma)\right]X+\left[\sigma-k+\frac{k(1-\gamma)}{\gamma}\right]XV-DX^2+XO(|(X,V)|^2),\end{array}\right.
\end{equation}
with
$$
D=\frac{k(1-\gamma)\alpha[(N+\sigma-k)\beta+\alpha p(1-\gamma)]}{\beta^2}.
$$
We omit here the detailed calculations leading to all the previous expressions, since they are rather tedious but straightforward. We next introduce a new independent variable by setting, for $X>0$,
$$
d\theta=X d\eta,
$$
mapping the reduced system \eqref{flowgamma} into the system obtained by dividing by $X$ its both equations. We then notice that $(X,V)=(0,0)$ remains a critical point with respect to this new variable and there are orbits connecting to it if and only if the linear term in the equation for $\dot{V}$ vanishes, that is
$$
\sigma\gamma+k(1-\gamma)=\frac{\gamma}{\beta}\left[\sigma\beta+(1-\gamma)\alpha(p-1)\right]=0,
$$
which leads to either $\gamma=0$ (not of interest here) or $\gamma=\gamma_0$ defined in \eqref{eq.gamma}. Let us fix now $\gamma=\gamma_0$. Then the reduced system \eqref{flowgamma} becomes with respect to the new variable $\theta$
\begin{equation}\label{flowgamma2}
\left\{\begin{array}{ll}\dot{X}=-\frac{\sigma(m-1)+2(p-1)}{p-1}X+XO(|(X,V)|),\\
\dot{V}=-\frac{1}{\beta}V-DX+O(|(X,V)|^2),\end{array}\right.
\end{equation}
thus the origin of it is a stable node. It follows that the center manifold is in fact stable and thus the point $P_{\gamma_0}$ behaves like an attractor for all orbits coming from the half-space $\{X>0\}$ according to \cite[Lemma 1, Section 2.4]{Carr}, since also the only nonzero eigenvalue is negative. The orbits entering this point contain profiles such that $Z\sim\gamma_0$, which leads to the tail behavior \eqref{queue} after undoing the change of variable \eqref{PSchange}. Moreover, since $X\to0$ at $P_{\gamma_0}$, we also infer that this behavior is taken as $\xi\to\infty$. For $\gamma\neq\gamma_0$, a simple integration of \eqref{flowgamma} in a neighborhood of $P_{\gamma}$ and in the half-space $\{X>0\}$ leads in a first approximation to
$$
\frac{dV}{dX}\sim\frac{C}{X}, \qquad C=\sigma\gamma+k(1-\gamma)\neq0,
$$
whence $V\sim C\ln\,X$ and this is no longer an orbit passing through the origin $(X,V)=(0,0)$. We thus infer that there are no orbits either entering or going out of $P_{\gamma}$ with $\gamma\neq\gamma_0$ except for the ones fully included in the invariant plane $\{X=0\}$, which do not contain profiles.
\end{proof}

\noindent \textbf{Changes for $p=1$.} In the case $p=1$ there is a first change with respect to the critical point $P_1$, which now expands into a critical line
$$
P_1^{\gamma}=\left(0,-\frac{\beta}{\alpha},\gamma\right), \qquad \gamma>0.
$$
\begin{lemma}[Local analysis near $P_1^{\gamma}$ for $p=1$]\label{lem.P1p1}
For any $\gamma>0$, the critical point $P_1^{\gamma}$ has a one-dimensional stable manifold, a one-dimensional unstable manifold and a one-dimensional center manifold. The orbits entering $P_1^{\gamma}$ on the stable manifolds of these points contain profiles with interface behaving as in \eqref{beh.P1}, while the unstable manifolds are all contained in the invariant plane $\{X=0\}$ and the center manifold is unique for each $\gamma\in(0,\infty)$ and contained in the line $\{X=0, Y=-\beta/\alpha\}$.
\end{lemma}
We omit here the proof, as it is given in \cite[Lemma 2.2]{IS19}, where the interface point of the profile entering $P_1^{\gamma}$ is related to $\gamma$ by $\xi_0=(\alpha\gamma)^{1/\sigma}$.

Another significant change in the case $p=1$ comes from the analysis of the critical points $P_{\gamma}$. Indeed, $\gamma_0$ as in \eqref{eq.gamma} is no longer well defined for $p=1$, and the expectation is that the attractor moves to the infinity of the phase space, a fact already noticed in \cite[Lemma 2.10]{IS19} in dimension $N=1$. This specific case will be addressed in Section \ref{sec.infty}. This also means that, for $p=1$ and any $\gamma\in(0,\infty)$, the critical points $P_{\gamma}$ do not contain interesting orbits, according to the proof of Lemma \ref{lem.Pg}.

\section{Local analysis of the critical points at infinity}\label{sec.infty}

In order to understand all the possible local behaviors of the profiles, a local analysis of the critical points of the system \eqref{PSsyst1} lying at infinity is also required. For this analysis we follow the theory in \cite[Section 3.10]{Pe} by passing to the Poincar\'e hypersphere through the new variables $(\overline{X},\overline{Y},\overline{Z},W)$ defined as
$$
X=\frac{\overline{X}}{W}, \qquad Y=\frac{\overline{Y}}{W}, \qquad Z=\frac{\overline{Z}}{W}
$$
and obtaining that the critical points at space infinity solve the following system
\begin{equation}\label{Poincare}
\left\{\begin{array}{ll}\overline{X}(\overline{X}\overline{Z}+(N-2)\overline{X}\overline{Y}+m\overline{Y}^2)=0,\\
\overline{X}\overline{Z}[(\sigma+2)\overline{X}+(p-m)\overline{Y}]=0,\\
\overline{Z}[p\overline{Y}^2+(\sigma+N)\overline{X}\overline{Y}+\overline{X}\overline{Z}]=0,\end{array}\right.
\end{equation}
together with the condition of belonging to the equator of the hypersphere, which leads to the additional equation $\overline{X}^2+\overline{Y}^2+\overline{Z}^2=1$. Taking into account that we are considering only points with coordinates $\overline{X}\geq0$ and $\overline{Z}\geq0$ and that we are working in dimension $N\geq2$ (or $N=1$ with $\sigma>-1$), we find the following critical points on the Poincar\'e hypersphere:
\begin{equation*}
\begin{split}
&Q_1=(1,0,0,0), \ \ Q_{2,3}=(0,\pm1,0,0), \ \ Q_4=(0,0,1,0), \\
&Q_5=\left(\frac{m}{\sqrt{(N-2)^2+m^2}},-\frac{N-2}{\sqrt{(N-2)^2+m^2}},0,0\right).
\end{split}
\end{equation*}
For the main part of this section, and in order to avoid exceptional cases, we assume (unless specified something else) that $N\geq3$ and $p>1$. Under these assumptions, we analyze one by one the critical points below.

\medskip

\noindent \textbf{Local analysis near $Q_1$ and $Q_5$}. In order to analyze the flow of the system \eqref{PSsyst1} near these points, we introduce a new change of variable following \cite[Theorem 5(a), Section 3.10]{Pe} in order to translate them into the finite part of a new phase space topologically equivalent to the original one. We thus let
\begin{equation}\label{interm7}
y=\frac{Y}{X}, \qquad z=\frac{Z}{X}, \qquad w=\frac{1}{X},
\end{equation}
to obtain the system
\begin{equation}\label{systinf1}
\left\{\begin{array}{ll}\dot{y}=-(N-2)y-z+w-my^2-\frac{\beta}{\alpha}yw,\\
\dot{z}=(\sigma+2)z+(p-m)yz,\\
\dot{w}=2w-(m-1)yw,\end{array}\right.
\end{equation}
where the minus sign has been chosen in the general framework of \cite[Theorem 5, Section 3.10]{Pe}, since in the equation for $\dot{X}$ in the original system \eqref{PSsyst1} it occurs that $\dot{X}<0$ in a neighborhood of $Q_1$ (which is obvious since $|X/Y|\to+\infty$ near this point). We thus notice that in the new system \eqref{systinf1} the critical point $Q_1$ is mapped into its origin $(y,z,w)=(0,0,0)$ and the critical point $Q_5$ into the point $(y,z,w)=(-(N-2)/m,0,0)$.
\begin{lemma}[Local analysis near $Q_1$]\label{lem.Q1}
Let $N\geq3$. The system \eqref{systinf1} in a neihghborhood of $Q_1$ has a two-dimensional unstable manifold and a one-dimensional stable manifold. The orbits going out on the unstable manifold contain profiles with the local behavior \eqref{beh.Q1} as $\xi\to0$.
\end{lemma}
\begin{proof}
The linerization of the system \eqref{systinf1} in a neighborhood of $(y,z,w)=(0,0,0)$ has the matrix
$$
M(Q_1)=\left(
         \begin{array}{ccc}
           -(N-2) & -1 & 1 \\
           0 & \sigma+2 & 0 \\
           0 & 0 & 2 \\
         \end{array}
       \right),
$$
with two positive eigenvalues and one negative eigenvalue. The stable manifold is contained in the plane $\{w=0\}$, corresponding to $X=\infty$. The orbits going out on the unstable manifold satisfy $dz/dw\sim(\sigma+2)z/2w$, which after integration reads $z\sim Cw^{(\sigma+2)/2}$. By undoing the change of variable \eqref{interm7} we find that $Z\sim KX^{-\sigma/2}$, $K>0$, which in terms of profiles gives $f(\xi)\sim K$ for some $K>0$, and this asymptotic behavior is taken as $\xi\to0$ since $X\to\infty$ at $Q_1$. The analysis splits now into three cases with respect to the range of $\sigma$ as follows:

$\bullet$ if $\sigma<0$, we have $z/w\sim Cw^{\sigma/2}\to+\infty$ in a neighborhood of $Q_1$, thus in a first approximation the $w$ term in the first equation in \eqref{systinf1} is dominated by the $z$ term and the following asymptotic approximation holds true
$$
\frac{dy}{dz}\sim-\frac{N-2}{\sigma+2}\frac{y}{z}-\frac{1}{\sigma+2}.
$$
This can be integrated to deduce next that
\begin{equation}\label{interm9}
y\sim-\frac{z}{N+\sigma}+Kz^{-(N-2)/(\sigma+2)}.
\end{equation}
We have to take $K=0$ in \eqref{interm9}, since we are looking for orbits passing through the origin, hence $y/z=Y/Z\sim-1/(N+\sigma)$. Translating this in terms of profiles via \eqref{PSchange} and integrating the resulting differential equation, we are led to the local behavior given by \eqref{beh.Q1} with $\sigma<0$.

$\bullet$ if $\sigma>0$, we have $z/w\sim Cw^{\sigma/2}\to0$ in a neighborhood of $Q_1$, thus in a first approximation the $z$ term in the first equation in \eqref{systinf1} is dominated by the $w$ term and the following asymptotic approximation holds true
$$
\frac{dy}{dw}\sim-\frac{N-2}{2}\frac{y}{w}+\frac{1}{2},
$$
which gives after integration
\begin{equation}\label{interm8}
y\sim\frac{w}{N}+Kw^{-(N-2)/2}.
\end{equation}
We have to take again $K=0$ in \eqref{interm8}, as we are looking for orbits passing through the origin, hence $y/w=Y\sim 1/N$ in a neighborhood of $Q_1$. Recalling the definition of $Y$ in \eqref{PSchange} and integrating, we obtain the local behavior \eqref{beh.Q1} with $\sigma>0$.

$\bullet$ if $\sigma=0$, we observe that $z\sim Cw$ in a neighborhood of $Q_1$, and the constant $C>0$ is connected at the level of profiles with the value of $f(0)$. Indeed, we get $z/w=Z\to C$ as $\xi\to0$, and recalling that we are in the case $\sigma=0$, we get $f(0)=(\alpha C)^{1/(p-1)}$. We furthermore deduce from the first and third equations of the system \eqref{systinf1} and the local approximation $z=Cw$ that
$$
\frac{dy}{dw}\sim-\frac{N-2}{2}\frac{y}{w}+\frac{1-C}{2},
$$
which gives by integration
\begin{equation}\label{interm8bis}
y\sim\frac{(1-C)w}{N}+Kw^{-(N-2)/2}.
\end{equation}
We notice again that we have to take $K=0$ in \eqref{interm8bis} and then undoing the change of variable, we are left with $Y\sim(1-C)/N$. A final integration step gives the local behavior \eqref{beh.Q1} in this case. In particular, for $C=1$ we obtain a profile included in the plane $\{Y=0\}$ leading to the constant solution mentioned in the Introduction.
\end{proof}

\noindent \textbf{Remark.} Let us emphasize here that in a neighborhood of $Q_1$ and with $\sigma>0$, we have found that $Y\sim 1/N$. This information will be very important in the last section of the paper.

\begin{lemma}[Local analysis near $Q_5$]\label{lem.Q5}
Let $N\geq3$. The critical point $Q_5$ is an unstable node. The orbits going out of it into the finite part of the phase space contain profiles with a vertical asymptote at $\xi=0$ and the local behavior
\begin{equation}\label{beh.Q5}
f(\xi)\sim C\xi^{-(N-2)/m}, \qquad {\rm as} \ \xi\to0, \qquad C>0 \ {\rm free \ constant}.
\end{equation}
\end{lemma}
Since the local analysis near $Q_5$ does not depend on $\sigma$, the proof is completely identical to the one of \cite[Lemma 3.2]{IMS21}.

\medskip

\noindent \textbf{Bifurcation in dimension $N=2$}. We notice that the points $Q_1$ and $Q_5$ coincide in dimension $N=2$. Keeping for convenience the label $Q_1$ for this mixed point, its local analysis is different.
\begin{lemma}[Local analysis near $Q_1$ for $N=2$]\label{lem.Q15N2}
Let $N=2$. Then the critical point $Q_1$ is a saddle-node in the sense of the theory in \cite[Section 3.4]{GH}. There exists a two-dimensional unstable manifold on which the orbits contain good profiles with local behavior given by \eqref{beh.Q1} as $\xi\to0$. All the rest of the orbits going out of $Q_1$ contain profiles with a vertical asymptote at $\xi=0$ given by
\begin{equation}\label{beh.Q12}
f(\xi)\sim K\left(-\ln\,\xi\right)^{1/m}, \qquad {\rm as} \ \xi\to0, \qquad K>0 \ {\rm free \ constant}.
\end{equation}
\end{lemma}
\begin{proof}[Sketch of the proof]
It is immediate to see that at $N=2$ the critical point $Q_1=Q_5$ is a saddle-node and $M(Q_1)$ has eigenvalues $\lambda_1=0$, $\lambda_2=\sigma+2$ and $\lambda_3=2$. The two-dimensional unstable manifold tangent to the vector space spanned by the eigenvectors $e_2=(-1,\sigma+2,0)$ and $e_3=(1,0,2)$ (corresponding to the eigenvalues $\lambda_2$ and $\lambda_3$) contains orbits whose analysis leads to the local behavior \eqref{beh.Q1} exactly as in Lemma \ref{lem.Q1}. All the other orbits go out tangent to the direction of the eigenvector $e_1=(1,0,0)$ corresponding to the zero eigenvalue, according to the theory in \cite[Section 3.4]{GH}. This implies that $|y/w|\to\infty$ and $|y/z|\to\infty$ on these orbits in a neighborhood of $Q_1$. With the aid of this limit behavior, one can show in a first step that along these orbits one has $Y\to-\infty$ and $Y/Z\to-\infty$ in a neighborhood of the point $Q_1$ and then use this information to prove that the last three terms in Eq. \eqref{SSODE} are negligible with respect to the terms involving $(f^m)''$ and $(f^m)'$. Hence, the local behavior as $\xi\to0$ of the profiles contained in these orbits is given by equating
$$
(f^m)''(\xi)+\frac{1}{\xi}(f^m)'(\xi)\sim 0,
$$
leading to \eqref{beh.Q12} after an integration. A detailed proof is given in \cite[Lemma 3.5]{IMS22}.
\end{proof}
As a remark, if we think of the dimension $N$ as a parameter in the system \eqref{systinf1}, we are dealing with a \emph{transcritical bifurcation} of this system at $N=2$ in the sense of \cite{S73, GH}.

\medskip

\noindent \textbf{Differences for $N=1$ with $-1<\sigma<0$}. Letting $N=1$, one can easily observe that $Q_1$ becomes an unstable node (as an effect of the previously analyzed bifurcation at $N=2$) while $Q_5$ passes to have a two-dimensional unstable manifold. The orbits going out of $Q_1$ contain profiles with $f(0)=K>0$ and with any possible value for $f'(0)\in\real$, while the orbits going out of $Q_5$ contain profiles with the local behavior
\begin{equation}\label{beh.Q5N1}
f(\xi)\sim K\xi^{1/m}, \qquad {\rm as} \ \xi\to0,
\end{equation}
with $K>0$ arbitrary constant. This is proved in detail in \cite[Section 6]{IMS21}, where an analysis of the more complicated situation that appears in dimension $N=1$ and with $\sigma\in(-2,-1]$ (introducing a new bifurcation and a new critical point) is also performed.

\medskip

\noindent \textbf{Local analysis near $Q_2$ and $Q_3$}. We translate these points into the finite part of a phase space following \cite[Theorem 5(b), Section 3.10]{Pe}. We thus set
$$
x=\frac{X}{Y}, \qquad z=\frac{Z}{Y}, \qquad w=\frac{1}{Y}
$$
and obtain the new system in variables $(x,z,w)$
\begin{equation}\label{systinf2}
\left\{\begin{array}{ll}\pm\dot{x}=-mx-(N-2)x^2-\frac{\beta}{\alpha}xw+x^2w-x^2z,\\
\pm\dot{z}=-pz-\frac{\beta}{\alpha}zw-(N+\sigma)xz-xz^2+xzw,\\
\pm\dot{w}=-w-\frac{\beta}{\alpha}w^2+xw^2-Nxw-xzw,\end{array}\right.
\end{equation}
where the signs have to be chosen according to the direction of the flow as follows: in a neighborhood of $Q_2$ one has to choose the minus sign in \eqref{systinf2}, while in a neighborhood of $Q_3$ one has to choose the plus sign, since $\dot{Y}$ is negative near both $Q_2$ and $Q_3$ but the direction of the flow is reversed. We thus identify $Q_2$ with the origin of \eqref{systinf2} when taken the minus sign and $Q_3$ with the origin of \eqref{systinf2} when taken the plus sign. We next deduce that $Q_2$ is an unstable node, while $Q_3$ is a stable node and the local behavior is given in the next lemma.
\begin{lemma}[Local analysis near $Q_2$ and $Q_3$]\label{lem.Q23}
The orbits going out of $Q_2$ to the finite part of the phase space contain profiles $f(\xi)$ which change sign at some $\xi_0\in(0,\infty)$ in the sense that $f(\xi_0)=0$, $(f^m)'(\xi_0)>0$. The orbits entering the point $Q_3$ from the finite part of the phase space contain profiles $f(\xi)$ which change sign at some $\xi_0\in(0,\infty)$ in the sense that $f(\xi_0)=0$, $(f^m)'(\xi_0)<0$.
\end{lemma}
We omit here the proof of this Lemma, since it follows the proof of the similar Lemma in previous works such as for example \cite[Lemma 2.6]{IS21a}.

\medskip

\noindent \textbf{No connections to or from $Q_4$ for $p>1$}. We are left with the analysis of the critical point $Q_4$. The general theory is not helpful for this study, since an intent to use \cite[Theorem 5(b), Section 3.10]{Pe} in order to translate this point to the origin of an equivalent phase space produces a critical point whose linearization has all the eigenvalues equal to zero and an analysis of it would be extremely involved. We thus prove that there are no interesting orbits connecting (in any way) to this point in a direct way, working with Eq. \eqref{SSODE}.
\begin{lemma}\label{lem.Q4}
Let $p>1$. Then there are no profiles $f(\xi)$ contained in any orbit either going out or entering the critical point $Q_4$.
\end{lemma}
\begin{proof}
Assume for contradiction that there are orbits connecting to $Q_4$ and containing solutions to Eq. \eqref{SSODE}. On such an orbit, the component $Z$ is dominating, thus we have $Z\to\infty$, $X/Z\to0$ and $Y/Z\to0$ when approaching $Q_4$. These limits read in terms of profiles
\begin{equation}\label{interm10}
\xi^{\sigma}f(\xi)^{p-1}\to\infty, \qquad \xi^{-2-\sigma}f(\xi)^{m-p}\to0, \qquad \xi^{-1-\sigma}f(\xi)^{m-p-1}f'(\xi)\to0.
\end{equation}
The limits in \eqref{interm10} can be taken simultaneously either as $\xi\to0$, or as $\xi\to\infty$, or as a third option as $\xi\to\xi_0\in(0,\infty)$. We will rule out one by one all these possibilities below.

\medskip

\noindent \textbf{Case 1: $\xi\to0$}. If $\sigma\geq0$, we infer on the one hand from the first limit in \eqref{interm10} and the fact that $p>1$ that $f(\xi)\to\infty$. On the other hand, the second limit in \eqref{interm10} readily gives that $f(\xi)\to0$, since $m>p$, and we reach a contradiction. If $\sigma<0$, we again infer from the first two limits in \eqref{interm10} that there exists $\epsilon\in(0,1)$ sufficiently small such that
$$
\xi^{-\sigma/(p-1)}<f(\xi)<\xi^{(\sigma+2)/(m-p)}, \qquad {\rm for} \ \xi\in(0,\epsilon).
$$
This in particular implies that
$$
\xi^{-\sigma/(p-1)}<\xi^{(\sigma+2)/(m-p)}
$$
and, taking into account that $\xi\in(0,1)$, we further find that
$$
-\frac{\sigma}{p-1}<\frac{\sigma+2}{m-p},
$$
which leads to a contradiction with the fact that in our range \eqref{range.exp} we have $\sigma(m-1)+2(p-1)>0$.

\medskip

\noindent \textbf{Case 2: $\xi\to\xi_0\in(0,\infty)$}. This is very easy to be ruled out, since we infer immediately from the first two limits in \eqref{interm10} and the fact that $\xi\to\xi_0$ that on the one hand $f(\xi)^{p-1}\to\infty$ and on the other hand $f(\xi)^{m-p}\to0$, leading to a contradiction since both $p-1$ and $m-p$ are positive.

\medskip

\noindent \textbf{Case 3: $\xi\to\infty$}. This last case is more involved, and its discussion will be split into several steps for the reader's convenience.

\medskip

\noindent \textbf{Case 3, Step 1.} We prove first that there exists $R>0$ sufficiently large such that $f(\xi)$ is monotone on $(R,+\infty)$. Assume for contradiction that there exists a sequence of local minima $\xi_{0,n}$ of the profile $f(\xi)$ such that $\xi_{0,n}\to\infty$ as $n\to\infty$. We then have $f'(\xi_{0,n})=0$ and $(f^m)'(\xi_{0,n})=0$ for any $n$, and also
$$
(f^m)''(\xi_{0,n})=m(m-1)f(\xi_{0,n})^{m-2}f'(\xi_{0,n})^2+mf(\xi_{0,n})^{m-1}f''(\xi_{0,n})\geq0.
$$
By evaluating \eqref{SSODE} at $\xi=\xi_{0,n}$ we find
$$
f(\xi_{0,n})(\xi_{0,n}^{\sigma}f(\xi_{0,n})^{p-1}-\alpha)=-(f^m)''(\xi_{0,n})\leq0,
$$
which is a contradiction with the fact that
$$
\lim\limits_{n\to\infty}\xi_{0,n}^{\sigma}f(\xi_{0,n})^{p-1}=\infty.
$$
It thus follows that the profile $f(\xi)$ must be monotone on some interval $(R,+\infty)$ and thus have a limit as $\xi\to\infty$, which we denote by $L\in[0,\infty]$.

\medskip

\noindent \textbf{Case 3, Step 2}. Assume now for contradiction that $L=\lim\limits_{\xi\to\infty}f(\xi)\in(0,\infty)$. We apply twice a calculus result stated rigorously as \cite[Lemma 2.9]{IL13}, in a first step to the function
$$
h(\xi)=\left\{\begin{array}{ll}f(\xi)-L, & {\rm if} \ f \ {\rm is \ increasing},\\L-f(\xi), & {\rm if} \ f \ {\rm is \ decreasing},\end{array}\right.
$$
in order to deduce that there is a subsequence $\xi_{k}$ such that $\xi_{k}\to\infty$ and $\xi_kh'(\xi_k)\to0$ as $k\to\infty$, and then to the function $\xi h'(\xi)$, showing that there exists a subsequence (relabeled as $\xi_k$ for simplicity) such that at the same time we have
$$
\lim\limits_{k\to\infty}\xi_k=+\infty, \qquad \lim\limits_{k\to\infty}h(\xi_k)=\lim\limits_{k\to\infty}\xi_kh'(\xi_k)=\lim\limits_{k\to\infty}\xi_k^2h''(\xi_k)=0,
$$
which gives in terms of the initial function $f$ that
\begin{equation}\label{interm11}
\lim\limits_{k\to\infty}f(\xi_k)=L, \qquad \lim\limits_{k\to\infty}\xi_kf'(\xi_k)=\lim\limits_{k\to\infty}\xi_k^2f''(\xi_k)=0.
\end{equation}
In particular, we also infer from \eqref{interm11} that
$$
\lim\limits_{k\to\infty}(f^m)''(\xi_k)=\lim\limits_{k\to\infty}\frac{N-1}{\xi_k}(f^m)'(\xi_k)=0,
$$
whence, by evaluating \eqref{SSODE} at $\xi=\xi_k$, we are left with
$$
\lim\limits_{k\to\infty}f(\xi_k)[\xi_k^{\sigma}f(\xi_k)^{p-1}-\alpha]=0,
$$
which is a contradiction to the first limit in \eqref{interm10}.

\medskip

\noindent \textbf{Case 3, Step 3}. Assume now for contradiction that $\lim\limits_{\xi\to\infty}f(\xi)=\infty$. A previous step in the proof gives that $f$ is in this case increasing on some interval $(R,\infty)$, that is, $f'(\xi)>0$ and $(f^m)'(\xi)>0$ for any $\xi\in(R,\infty)$. We then infer from \eqref{SSODE} and the first limit in \eqref{interm10} that
$$
(f^m)''(\xi)\leq f(\xi)[\alpha-\xi^{\sigma}f(\xi)^{p-1}]\to-\infty \qquad {\rm as} \ \xi\to\infty,
$$
for any $\xi\in(R,\infty)$. There exists thus some $R_1>R$ such that $(f^m)''(\xi)<-1$ for any $\xi\in[R_1,\infty)$. The mean-value theorem then gives
$$
(f^m)'(\xi)<(f^m)'(R_1)-(\xi-R_1), \qquad {\rm for \ any} \ \xi>R_1,
$$
and by letting $\xi$ very large we reach a contradiction with the fact that $(f^m)'(\xi)>0$ for any $\xi\in(R_1,\infty)$.

\medskip

\noindent \textbf{Case 3, Step 4}. Assume now for contradiction that $\lim\limits_{\xi\to\infty}f(\xi)=0$, and from the previous steps we also obtain that $f$ is in this case decreasing on some interval $(R,\infty)$. If $\sigma\leq0$ we get that $\xi^{\sigma}f(\xi)^{p-1}\to0$ as $\xi\to\infty$, contradicting the first limit in \eqref{interm10}. Let now $\sigma>0$. We then write Eq. \eqref{SSODE} in the following form
\begin{equation}\label{interm12}
(f^m)''(\xi)+\frac{N-1}{\xi}(f^m)'(\xi)+\left[\frac{\xi^{\sigma}f(\xi)^{p-1}}{2}-\alpha\right]f(\xi)+\frac{\xi^{\sigma}f(\xi)^{p}}{2}+\beta\xi f'(\xi)=0.
\end{equation}
The first limit in \eqref{interm10} implies that
\begin{equation}\label{interm13}
\left[\frac{\xi^{\sigma}f(\xi)^{p-1}}{2}-\alpha\right]f(\xi)>0
\end{equation}
for $\xi$ sufficiently large, while the third limit in \eqref{interm10} gives that
\begin{equation}\label{interm14}
\begin{split}
\frac{N-1}{\xi}(f^m)'(\xi)&+\frac{\xi^{\sigma}f(\xi)^{p}}{2}+\beta\xi f'(\xi)\\&=\xi^{\sigma}f(\xi)^p\left[\frac{1}{4}+m(N-1)\xi^{-1-\sigma}f(\xi)^{m-p-1}f'(\xi)\right]\\
&+f(\xi)\left[\frac{1}{4}\xi^{\sigma}f(\xi)^{p-1}+\frac{\beta\xi f'(\xi)}{f(\xi)}\right]=f(\xi)(T_1+T_2),
\end{split}
\end{equation}
where
$$
T_1:=\xi^{\sigma}f(\xi)^{p-1}\left[\frac{1}{4}+m(N-1)\xi^{-1-\sigma}f(\xi)^{m-p-1}f'(\xi)\right]
$$
and
$$
T_2:=\frac{1}{4}\xi^{\sigma}f(\xi)^{p-1}+\frac{\beta\xi f'(\xi)}{f(\xi)}.
$$
We immediately infer from the first and third limits in \eqref{interm10} that $T_1(\xi)\to+\infty$ as $\xi\to\infty$, while a calculus exercise based on the first limit in \eqref{interm10} and the definition of the limit readily gives that $T_2$ is non-negative in a neighborhood of $+\infty$. We thus conclude from \eqref{interm12}, \eqref{interm13}, \eqref{interm14} and the previous considerations that $(f^m)''(\xi)<0$ in some interval $(R,\infty)$ for sufficiently large $R$, which is an obvious contradiction with the fact that the function $f^m$ has $y=0$ as a horizontal asymptote.
\end{proof}

\medskip

\noindent \textbf{The critical point $Q_4$ for $p=1$}. This is a special case which introduces important changes, in line with the analysis performed in dimension $N=1$ for the case $p=1$ in \cite[Lemma 2.10]{IS19}. We will generalize this analysis to the $N$-dimensional case below. Notice first that for $p=1$ we are restricted only to $\sigma>0$, according to \eqref{range.exp}.
\begin{lemma}\label{lem.Q4p1}
Let $p=1$ and $\sigma>0$. The critical point $Q_4$ behaves like a stable node for the orbits coming from the finite part of the phase space associated to the system \eqref{PSsyst1}. The orbits entering this critical point contain profiles presenting a tail at infinity with the decay rate
\begin{equation}\label{beh.Q4}
f(\xi)\sim K\xi^{(\sigma+2)/(m-1)}e^{-\xi^{\sigma}}, \qquad {\rm as} \ \xi\to\infty.
\end{equation}
\end{lemma}
\begin{proof}[Sketch of the proof]
Notice first that, since $p=1$, we have $Z=\xi^{\sigma}\to\infty$ on all the orbits connecting to $Q_4$, thus this point can be reached only by orbits entering it as $\xi\to\infty$. The proof of their local behavior follows very closely the proofs of \cite[Lemma 2.9 and Lemma 2.10]{IS19} to which we refer, and we only give here the differences with respect to the above mentioned proofs. In a first step, it is shown that if $f(\xi)$ is a solution to Eq. \eqref{SSODE} with $p=1$ and such that $f(\xi)>0$ for any $\xi>0$, then there exists $R>0$ and $K(R)>0$ such that
\begin{equation}\label{interm24}
f(\xi)\leq K(R)\xi^{(\sigma+2)/(m-1)}e^{-\xi^{\sigma}}, \qquad {\rm for \ any} \ \xi>R.
\end{equation}
In order to prove the estimate \eqref{interm24}, one can follow verbatim the proof of \cite[Lemma 2.9]{IS19} to show that necessarily a profile such that $f(\xi)>0$ for any $\xi>0$ has to decrease to zero in an interval $(R,\infty)$ for some $R>0$ large. Moreover, by an argument of contradiction with sequence of zeros also given in \cite[Lemma 2.9]{IS19}, one can also get that $(f^m)''(\xi)\geq0$ for any $\xi\in(R,\infty)$ (by taking $R$ larger in order to fulfill both conditions). We infer from \eqref{SSODE} that
\begin{equation}\label{interm25}
(\xi^{\sigma}-\alpha)f(\xi)+\beta\xi f'(\xi)+\frac{N-1}{\xi}(f^m)'(\xi)\leq0, \qquad \xi\in(R,\infty).
\end{equation}
Fixing $\epsilon>0$ sufficiently small, we can increase $R>0$ further in order to have
$$
-\frac{(N-1)m(f^{m-1})'(\xi)}{(m-1)\beta}<\epsilon, \qquad \xi>R,
$$
and the estimate \eqref{interm25} can (by dividing it by $\beta\xi f(\xi)$) be written in an equivalent form as
$$
\frac{f'(\xi)}{f(\xi)}\leq-\sigma\xi^{\sigma-1}+\frac{\sigma+2}{(m-1)\xi}+\frac{\epsilon}{\xi^2},
$$
which leads after an integration on $(\xi_0,\xi)$, $\xi_0>R$ fixed, to
$$
f(\xi)\leq K(\xi_0)\xi^{(\sigma+2)/(m-1)}e^{-\xi^{\sigma}}e^{\epsilon/\xi},
$$
for any $\xi>R$. The inequality \eqref{interm24} follows then by noticing that, for $\xi>R$ large enough, $e^{\epsilon/\xi}<2$ and thus doubling the constant $K(\xi_0)$. In a second step, we show that the local behavior \eqref{beh.Q4} is actually taken on orbits entering $Q_4$. To this end, we proceed as in \cite[Lemma 2.10]{IS19} by performing the change of variable $W=XZ$ to find a new dynamical system
\begin{equation}\label{PSsyst1bis}
\left\{\begin{array}{ll}\dot{X}=X[(m-1)Y-2X],\\
\dot{Y}=-Y^2-\frac{\beta}{\alpha}Y+X-NXY-W,\\
\dot{W}=W[(m-1)Y+(\sigma-2)X].\end{array}\right.
\end{equation}
Notice that with this change of variable the point $Q_4$ ``moves" at the origin. More precisely, we infer from the estimate \eqref{interm24} that any orbit entering $Q_4$ in the phase space associated to the system \eqref{PSsyst1} will now enter the origin of the system \eqref{PSsyst1bis} as
$$
W=XZ\sim\xi^{\sigma-2}f(\xi)^{m-1}\to0, \qquad {\rm as} \ \xi\to\infty.
$$
The local analysis of the flow of the system \eqref{PSsyst1bis} in a neighborhood of the origin is completely similar to the analysis performed in \cite[Lemma 2.10]{IS19} in dimension $N=1$, since a simple inspection of the proof shows that the term $NXY$ in the second equation of \eqref{PSsyst1bis} is completely irrelevant for the flow on the center manifold in a neighborhood of the origin. We thus refer the reader to the proof of \cite[Lemma 2.10]{IS19} for the rest of the argument.
\end{proof}

\section{Existence of good profiles with interface}\label{sec.exist}

This section is devoted to the proof of Theorem \ref{th.exist}, based on a shooting technique in backward sense from the interface point. Such a technique has been used also in the one-dimensional case with $\sigma>0$ \cite[Section 3]{IS21a}, but based on continuity arguments with respect to parameters directly for the solutions to Eq. \eqref{SSODE}. This approach is no longer valid here at least for $\sigma<0$ due to the singular coefficient at $\xi=0$, then we will perform all the shooting argument in the phase space. We are thus interested in monitoring where do the orbits entering the critical point $P_1$ come from.
\begin{proof}[Proof of Theorem \ref{th.exist}]
We divide the proof into several steps.

\medskip

\noindent \textbf{Step 1. The two-dimensional manifold.} As we know from Lemma \ref{lem.P1}, there exists a two-dimensional stable manifold of orbits entering $P_1$ generated by the eigenvalues corresponding to the first and third equation in the system \eqref{PSsyst1}. Since $P_1$ is a hyperbolic critical point, the Hartman-Grobman theorem implies that the orbits entering $P_1$ on the stable manifold are tangent to the directions of the system obtained by keeping only the linear terms. Thus, recalling that $Y=-\beta/\alpha$ at $P_1$, we can write
$$
\frac{dZ}{dX}\sim\frac{(p-1)YZ}{(m-1)XY}=\frac{p-1}{m-1}\frac{Z}{X},
$$
thus the orbits enter $P_1$ tangent to the curves obtained by integration, that is
\begin{equation}\label{interm15}
Z=KX^{(p-1)/(m-1)}, \qquad K\in(0,\infty),
\end{equation}
having thus two limits: one included in the invariant plane $\{Z=0\}$ corresponding to $K=0$ and one included in the invariant plane $\{X=0\}$ corresponding to the limit $K\to\infty$.

\medskip

\noindent \textbf{Step 2. Limit in the plane $\{Z=0\}$}. When restricted to the invariant plane $\{Z=0\}$, the system \eqref{PSsyst1} reduces to
\begin{equation}\label{PSsyst1Z0}
\left\{\begin{array}{ll}\dot{X}=X[(m-1)Y-2X],\\
\dot{Y}=-Y^2-\frac{\beta}{\alpha}Y+X-NXY,\end{array}\right.
\end{equation}
and it is easy to see that the critical point $P_1=(0,-\beta/\alpha)$ is a saddle point, thus there exists a unique orbit entering this point in the phase plane associated to the system \eqref{PSsyst1Z0}. We show that this orbit remains always negative and in fact it stays in the half-plane $\{Y\leq-\beta/\alpha\}$. To this end, we first notice that the flow of the system \eqref{PSsyst1Z0} on the line $\{Y=-\beta/\alpha\}$ is given by the sign of the expression
$$
X\left(1+\frac{N\beta}{\alpha}\right)>0,
$$
thus this line can be only crossed from left to right (in the positive direction with respect to $Y$) by orbits of the system. But we know that the orbit entering $P_1$ arrives tangent to the eigenvector corresponding to the negative eigenvalue $\lambda_1=-(m-1)\beta/\alpha$ of the critical point $P_1$, which is
$$
e_1=(m\beta,-(\alpha+N\beta)),
$$
hence it enters the point from the half-plane $\{Y<-\beta/\alpha\}$ and thus it lies forever in this half-plane. Coming back to the global analysis of the space, it follows that this orbit comes either from the unstable node $Q_5$ (for $N\geq3$) or from the node-sector of the saddle-node $Q_1$ (for $N=2$).

\medskip

\noindent \textbf{Step 3. Limit in the plane $\{X=0\}$}. Let us introduce the change of variable $\overline{Z}=XZ$ in the system \eqref{PSsyst1}, before restricting ourselves to the plane $\{X=0\}$, similarly to the analysis in \cite[Proposition 3.4]{IS21a}. Letting then $X=0$ in the newly obtained system, we are left with the following reduced system
\begin{equation}\label{PSsyst1X0}
\left\{\begin{array}{ll}
\dot{Y}=-Y^2-\frac{\beta}{\alpha}Y-\overline{Z},\\
\dot{\overline{Z}}=(m+p-2)Y\overline{Z},\end{array}\right.
\end{equation}
which is (modulo a rescaling of the coefficients) exactly the same system as in Step 1 of \cite[Proposition 3.4]{IS21a}, as it does not depend on the dimension $N$. The analysis of this system has been performed in the proof of the above quoted result, to which we refer for details. Its outcome is that there exists a unique orbit of the system \eqref{PSsyst1X0} entering the saddle point $P_1$, and this unique orbit comes from the critical point $Q_2$.

\medskip

\noindent \textbf{Step 4. The three-sets argument}. Recalling that all the orbits entering $P_1$ on the stable manifold are tangent to the one-parameter family of curves in \eqref{interm15}, we define the following three sets
\begin{equation*}
\begin{split}
&\mathcal{A}=\{K\in(0,\infty): {\rm the \ orbit \ with \ parameter} \ K \ {\rm in \ \eqref{interm15} \ comes \ from} \ Q_5\},\\
&\mathcal{C}=\{K\in(0,\infty): {\rm the \ orbit \ with \ parameter} \ K \ {\rm in \ \eqref{interm15} \ comes \ from} \ Q_2\},\\
&\mathcal{B}=\{K\in(0,\infty): {\rm the \ orbit \ with \ parameter} \ K \ {\rm in \ \eqref{interm15} \ does \ neither \ come \ from} \ Q_5 \ {\rm nor} \ Q_2\},
\end{split}
\end{equation*}
with the obvious adaptation of the node-sector of $Q_1$ instead of $Q_5$ if $N=2$. Since both $Q_5$ and $Q_2$ are unstable nodes, the sets $\mathcal{A}$ and $\mathcal{C}$ are both open. The orbit entering $P_1$ and contained in the invariant plane $\{Z=0\}$ comes from $Q_5$ which is an unstable node, thus we infer from standard continuity arguments that $\mathcal{A}$ is non-empty and contains an interval of the form $(0,K_*)$ for some $K_*>0$. A similar argument using the orbit included in the plane $\{X=0\}$ coming from $Q_2$ proves that $\mathcal{C}$ is also non-empty and contains an interval of the form $(K^*,\infty)$ for some $K^*>K_*>0$. We deduce that the set $\mathcal{B}$ is non-empty (and closed) by standard topology. Thus, there exists at least a parameter $K\in\mathcal{B}$. The orbit tangent to the curves in \eqref{interm15} with parameters $K\in\mathcal{B}$ cannot come from either $Q_5$ or $Q_2$, thus they should go out from one of the remaining critical points $Q_1$, $P_2$ or $P_0$, or from an $\alpha$-limit set.

\medskip

\noindent \textbf{Step 5. End of the proof}. We are left with ruling out the possibility of an $\alpha$-limit set as the origin of the orbits entering $P_1$ with parameters $K\in\mathcal{B}$. To this end, we derive from \cite[Theorem 1, Section 3.2]{Pe} that any $\alpha$-limit set has to be a compact set in the phase space. We show first that a profile $f(\xi)$ solution to Eq \eqref{SSODE} and contained in an orbit starting from an $\alpha$-limit set may only have damped oscillations. Assume for contradiction that this is not true. We can thus extract convergent subsequences of minima, respectively maxima, of $f(\xi)$
$$
\xi_{n,\min}\to\xi_{\min}, \qquad \xi_{n,\max}\to\xi_{\max}\in[0,\infty), \qquad {\rm as} \ n\to\infty,
$$
such that their terms are alternated (that is, a maximum point lies between two minima and viceversa) and $\xi_{\min}\neq\xi_{\max}$. We then deduce that there exist points
$$
\xi_n\in(\xi_{n,\max},\xi_{n,\min}), \qquad (f^m)''(\xi_n)=0, \qquad \xi_n\to\xi_0,
$$
where the latter is obtained by eventually restricting ourselves to a subsequence. Evaluating Eq. \eqref{SSODE} at $\xi=\xi_n$ we get
\begin{equation}\label{interm16}
\lim\limits_{n\to\infty}\left[\frac{m(N-1)}{\xi_n}f(\xi_n)^{m-1}+\beta\xi_n\right]f'(\xi_n)=\alpha f(\xi_0)-\xi_0^{\sigma}f(\xi_0)^p\in\real.
\end{equation}
Since $\xi_{\min}\neq\xi_{\max}$, we readily obtain from the mean-value theorem that $f'(\xi_n)\to\pm\infty$ and $(f^m)'(\xi_n)\to\pm\infty$ as $n\to\infty$ (the sign depending on whether the function oscillates from minima to maxima or viceversa when passing through $\xi_n$) and this leads to a contradiction with the limit in \eqref{interm16}. We find that the oscillations that a profile $f(\xi)$ may present are either finite or damped. 

Let us now go back to the phase space associated to the system \eqref{PSsyst1} and assume for contradiction that there exists an orbit entering $P_1$ with parameter $K\in\mathcal{B}$ and starting from an $\alpha$-limit set which is not reduced to a critical point. It is easy to show that, if both $X$ and $Z$ have a limit as $\eta\to-\infty$ along this orbit and only the $Y$ coordinate oscillates in a compact interval $[a,b]$, then the extremal points with $Y=a$ and $Y=b$ (together with the limits of $X$ and $Z$) are critical points. A detailed argument, based on subsequences of maxima and minima of $Y$ on the trajectory converging to $a$ and $b$, can be found in \cite[Proposition 4.10]{ILS22}. But this is impossible, as there are no different critical points in the system, with different finite values of $Y$ and unstable local behavior, hence the $\alpha$-limit is reduced to a single point. We are now left with the following two possibilities:

$\bullet$ There exist $X_{\min}<X_{\max}\in[0,\infty)$ such that 
$$
\liminf\limits_{\eta\to-\infty}X(\eta)=X_{\min}, \qquad \limsup\limits_{\eta\to-\infty}X(\eta)=X_{\max}.
$$
Expressing the above in terms of profiles and going back to the independent variable $\xi$, it follows that 
\begin{equation}\label{interm16bis}
\frac{\alpha}{m}X_{\min}\xi^{2}\leq f(\xi)^{m-1}\leq \frac{\alpha}{m}X_{\max}\xi^2,
\end{equation}
and the extremal values in \eqref{interm16bis} are taken along subsequences. We then conclude from the fact that oscillations must be damped in terms of $f(\xi)$ that the $\alpha$-limit is taken as $\xi\to0$. This fact, together with the definition of $Z$ in \eqref{PSchange} and the range of $\sigma$ in \eqref{range.exp}, readily implies that $Z\to0$ along the orbit going out of the $\alpha$-limit. The invariance of the plane $\{Z=0\}$ then entails that the compact $\alpha$-limit set lies in the plane $\{Z=0\}$, and we further infer from the Poincar\'e-Bendixon's Theorem \cite[Theorem 1, Section 3.7]{Pe} that the $\alpha$-limit set must be either a closed union of orbits between finite critical points (which is obviously impossible) or a periodic orbit inside the invariant plane $\{Z=0\}$. We prove next that this is not possible by using the Dulac's Criteria \cite[Theorem 2, Section 3.9]{Pe}. Indeed, if we restrict ourselves to the invariant plane $\{Z=0\}$, where the system \eqref{PSsyst1} reduces to \eqref{PSsyst1Z0} and we set $\theta=(3-m)/(m-1)$, we find by direct calculation that the divergence of the vector field of the system \eqref{PSsyst1Z0} multiplied by $X^{\theta}$, namely
\begin{equation*}
\begin{split}
{\rm div}X^{\theta}(\dot{X},\dot{Y})=-\frac{\beta}{\alpha}X^{\theta}-\left[N+\frac{2(m+1)}{m-1}\right]X^{\theta+1}
\end{split}
\end{equation*}
is strictly negative on all the half-plane $\{X>0, Z=0\}$, thus no periodic orbits may exist in this region.

$\bullet$ There exist $Z_{\min}<Z_{\max}\in[0,\infty)$ such that
$$
\liminf\limits_{\eta\to-\infty}Z(\eta)=Z_{\min}, \qquad \limsup\limits_{\eta\to-\infty}Z(\eta)=Z_{\max}.
$$
Expressing the above in terms of profiles and going back to the independent variable $\xi$, it follows that
\begin{equation}\label{interm16bisbis}
\alpha Z_{\min}\xi^{-\sigma}\leq f(\xi)^{p-1}\leq \alpha Z_{\max}\xi^{-\sigma},
\end{equation}
and the extremal values in \eqref{interm16bisbis} are taken along subsequences. We then conclude from the fact that oscillations must be damped in terms of $f(\xi)$ that the $\alpha$-limit is taken as $\xi\to0$ (and it may only happen when $\sigma<0$). This fact, together with the definition of $X$ in \eqref{PSchange} and the range of $\sigma$ in \eqref{range.exp}, readily implies that $X\to\infty$ along the orbit going out of the $\alpha$-limit. We can thus translate ourselves to the system \eqref{systinf1} which maps the limit $X\to\infty$ onto $w\to0$ and argument as above to conclude from the Poincar\'e-Bendixon's Theorem that there should be either a countable union of closed orbits between finitely many critical points or a periodic orbit inside the invariant plane $\{w=0\}$ of the system \eqref{systinf1}. But the system \eqref{systinf1} reduces in the plane $\{w=0\}$ to
\begin{equation}\label{systinf1bisbis}
\left\{\begin{array}{ll}\dot{y}=-(N-2)y-z-my^2,\\
\dot{z}=(\sigma+2)z+(p-m)yz,\end{array}\right.
\end{equation}
which has no critical points in the half-plane $\{z>0\}$ of it. Since, according to \cite[Theorem 5, Section 3.7]{Pe}, any periodic orbit must contain a critical point in the interior region of it, both a closed union of orbits between critical points or a periodic orbit cannot exist and thus the $\alpha$-limit is reduced to a point, as desired.
\end{proof}

\section{Classification of the profiles: $\sigma$ small}\label{sec.small}

In this section we prove the first part of Theorem \ref{th.class}, the rest of its proof being completed in the next Section \ref{sec.large}. We will thus show in this section that there exists $\sigma_0>0$ such that for any $\sigma\in(-2,\sigma_0)$, the orbits going out of both $P_2$ and $P_0$ cannot reach the critical point $P_1$ (and in fact they have to enter the critical point $P_{\gamma}$). This proof as a whole is probably the most technical one in the present work, thus, before entering into precise details, we will explain here its strategy for a better understanding. We will change our look during this proof by fixing $\sigma$ and moving $p$, more precisely letting
$$
-2<\sigma\leq0, \qquad 1-\frac{\sigma(m-1)}{2}<p<m.
$$
The main difficulty, as we see in Figure \ref{fig1}, is that for $p$ closer to 1, the orbits going out of $P_2$ enter the critical point $P_{\gamma}$ in a monotone way, while as $p$ approaches $m$, oscillations of them start to occur.

\begin{figure}[ht!]
  \begin{center}
  \subfigure[$p \ closer \ to \ 1$]{\includegraphics[width=7.5cm,height=6cm]{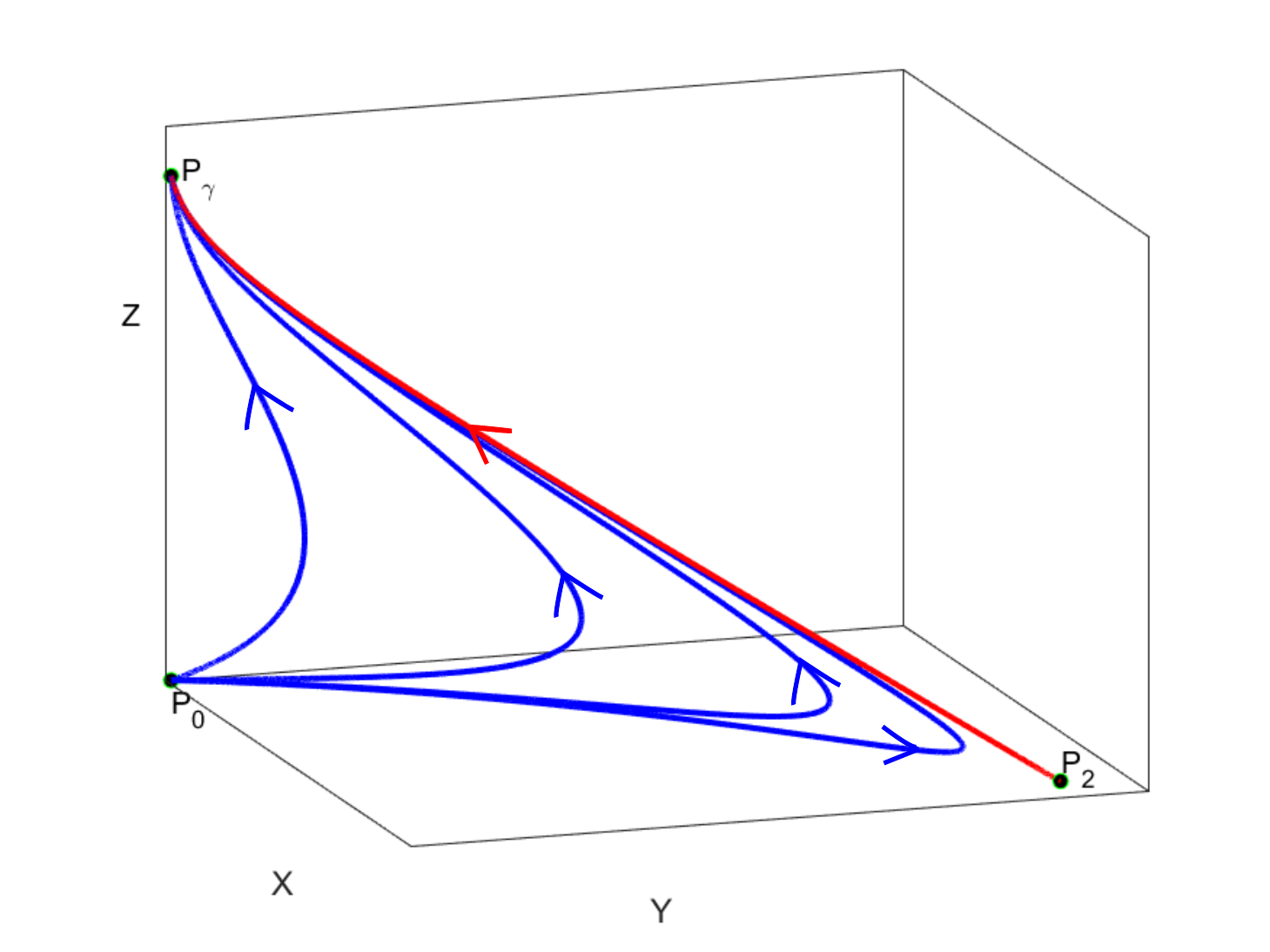}}
  \subfigure[$p \ closer \ to \ m$]{\includegraphics[width=7.5cm,height=6cm]{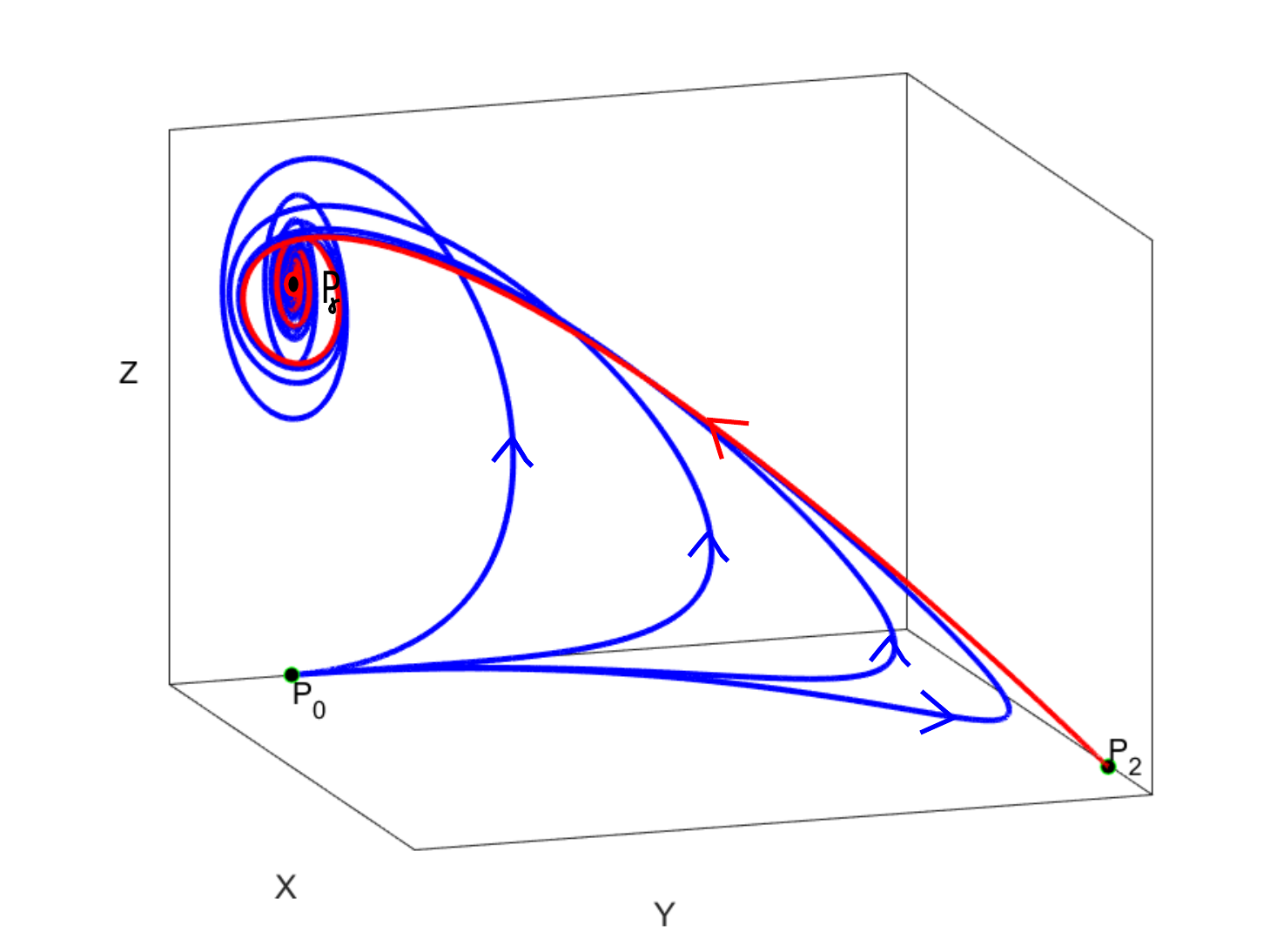}}
  \end{center}
  \caption{Orbits from $P_2$ going to $P_{\gamma}$ for $p$ is close to 1 and $p$ close to $m$. Experiments for $m=3$, $N=3$, $\sigma=0.2$ and $p=2$, respectively $p=2.99$.}\label{fig1}
\end{figure}

Owing to this fact, we are unable to cover the whole interval of $p$ with a single geometrical construction in the phase space, thus our strategy in the proof is \emph{tending a bridge} from $p\sim1$ to $p\sim m$. More precisely, when $p$ is small, we will employ as barrier for the flow in the phase space associated to the system \eqref{PSsyst1} a similar construction to the one that proved successful when dealing with $p\leq 1$ in \cite{IS19, IMS22}, while for $p$ closer to $m$, we will use a different and much more involved construction, limiting the orbits by means of a special surface that was very successful exactly in the limit case $p=m$, see \cite{IS22}. The joining point of the ``bridge" built with the two estimates from both sides lies at the following exponent
\begin{equation}\label{branch}
p_c(\sigma):=\frac{mN+\sigma+2}{N+\sigma+2}.
\end{equation}
It is now the moment to go to the detailed proof, which starts with a general, preparatory lemma.
\begin{lemma}\label{lem.walls}
At any point different from $P_2$ lying on any of the orbits going out of $P_2$ and of $P_0$ in the phase space associated to the dynamical system \eqref{PSsyst1}, it holds true that $X<X(P_2)$ and $Y<Y(P_2)$.
\end{lemma}
\begin{proof}
Let us consider the region $\mathcal{R}_1=\{X<X(P_2),Y<Y(P_2)\}$ in the phase space associated to the system \eqref{PSsyst1}. On the one hand, the direction of the flow of the system across the plane $\{X=X(P_2)\}$ is given by the sign of the expression
$$
X(P_2)[(m-1)Y-2X(P_2)]<0,
$$
provided $Y<Y(P_2)$. On the other hand, the direction of the flow of the system across the plane $\{Y=Y(P_2)\}$ is given by the sign of the expression
\begin{equation*}
F(X,Z)=-Y(P_2)^2-\frac{\beta}{\alpha}Y(P_2)+X-NXY(P_2)-XZ<X[1-NY(P_2)]-Y(P_2)\left[Y(P_2)+\frac{\beta}{\alpha}\right]<0,
\end{equation*}
provided $X<X(P_2)$. This shows that an orbit entering the region $\mathcal{R}_1$ cannot go out of it later on. Since the unique orbit going out of $P_2$ follows the direction of the eigenvector $e_3$ with components given by \eqref{vector.P2}, it goes into the interior of the region $\mathcal{R}_1$ and will stay there forever. The same happens in a more obvious way to all the orbits going out of $P_0$.
\end{proof}
The next step in the proof is the first part of the ``bridge", that is, to prove that for $p\in(1,p_c(\sigma))$ the orbits starting from $P_2$ and $P_{0}$ cannot reach $P_1$.
\begin{lemma}\label{lem.rooftop}
Let $N\geq2$. There exists $\sigma_{s}>0$ such that for any $\sigma\in(-2,\sigma_{s})$ and any $p\in(1-\sigma(m-1)/2,p_c(\sigma)]$, the orbits going out of $P_0$ and of $P_2$ cannot enter the critical point $P_1$.
\end{lemma}
\begin{proof}
The proof is divided into several steps for easiness. Notice that, in what follows, the fact that $N+\sigma>0$ is fundamental, thus we have to let $N\geq2$.

\medskip

\noindent \textbf{Step 1. Analysis for $\sigma\in(-2,0]$}. We already know from Lemma \ref{lem.walls} that $X<X(P_2)$ and $Y<Y(P_2)$ on these orbits. We next consider the plane
\begin{equation}\label{plane.roof}
Y+\frac{1}{N+\sigma}Z=\frac{1}{N}.
\end{equation}
in the phase space associated to the system \eqref{PSsyst1}. The direction of the flow of the system across this plane is given by the sign of the expression
\begin{equation}\label{interm17}
G(Z)=-\frac{p}{(N+\sigma)^2}Z^2+\frac{N(m-p)+(\sigma+2)(p+2)}{N(N+\sigma)(\sigma+2)}Z-\frac{N(m-p)+\sigma+2}{\sigma+2},
\end{equation}
which is a parabola with negative dominating coefficient and having the following roots
$$
Z_1=\frac{N+\sigma}{N}, \qquad Z_2=\frac{(N+\sigma)[N(m-p)+\sigma+2]}{Np(\sigma+2)}.
$$
We readily notice that, on the one hand, the root $Z_1$ corresponds in the plane \eqref{plane.roof} exactly to $Y=0$, and on the other hand, that $Z_1\leq Z_2$ provided $1<p\leq p_c(\sigma)$, where $p_c(\sigma)$ has been defined in \eqref{branch}. It follows that $G(Z)\leq0$ for $Y\geq0$ and $1<p\leq p_c(\sigma)$ and thus the flow of the system \eqref{PSsyst1} through the plane \eqref{plane.roof} has negative direction. Taking into account that $Y(P_2)<1/N$ by \eqref{ineq1}, it follows that the orbit going out of $P_2$ (and the same happens with all the orbits going out of $P_0$, as they enter first the half-space $\{Y>0\}$ as proved in Lemma \ref{lem.P0}) starts in the half-space $\{Y+Z/(N+\sigma)<1/N\}$ and it will remain there at least until it crosses the plane $\{Y=0\}$ as it cannot cross the plane \eqref{plane.roof} from below. In particular, this implies that if this orbit crosses the plane $\{Y=0\}$, it must do it at a point with coordinate $Z<(N+\sigma)/N$. We notice that, for $\sigma\in(-2,0]$, $(N+\sigma)/N\leq1$, and since the plane $\{Y=0\}$ can be crossed only in the region $\{Z>1\}$ according to the direction of the flow on it (given by the sign of $X(1-Z)$), it follows that all the orbits going out of $P_0$ and $P_2$ \emph{will remain forever in the half-space $\{Y\geq0\}$} if $\sigma\leq0$.

\medskip

\noindent \textbf{Step 2. $\sigma>0$ small}. The next step in the proof is to consider $\sigma>0$ but sufficiently small and work in the region $\{Y<0\}$. Consider the surface of equation
\begin{equation}\label{interm19}
X(Z-1)=k_1^2, \qquad k_1=\frac{\beta}{2\alpha}.
\end{equation}
The normal vector to this surface is $\overline{n}=(Z-1,0,X)$ and the direction of the flow of the system \eqref{PSsyst1} on the surface \eqref{interm19} is given by the sign of
\begin{equation*}
H(X,Y)=(\dot{X},\dot{Y},\dot{Z})\cdot\overline{n}=\left[\sigma X+\frac{(\sigma-2)\beta^2}{4\alpha^2}\right]X+\left[(p-1)X+\frac{(m+p-2)\beta^2}{4\alpha^2}\right]Y.
\end{equation*}
From now on, we will always compute the flow across planes or surfaces as the scalar product between the vector field of the system and the normal vector. Since we work now in the region $\{Y\leq0\}$, it is easy to see that the second term in the expression of $H(X,Y)$ is always negative, while the first term in the expression of $H(X,Y)$ is also negative provided $X<X(P_2)$ and
\begin{equation}\label{interm19bis}
0<\sigma<\frac{2\beta^2}{4\alpha^2X(P_2)+\beta^2},
\end{equation}
an inequality that is compatible on some interval $\sigma\in(0,\sigma_{s})$ despite the fact that the right-hand side also depends on $\sigma$ but has a positive minimum on some interval $(0,\sigma_s)$ with $\sigma_s>0$ sufficiently small. Thus, for such interval of $\sigma$, the flow on the surface \eqref{interm19} has negative direction. Moreover, since we have shown in the previous step that any orbit going out of $P_0$ or $P_2$ has to cross the plane $\{Y=0\}$ at a height $Z\leq(N+\sigma)/N$, that is $Z-1\leq\sigma/N$, it follows that if we let $\sigma>0$ sufficiently small such that it satisfies the estimate \eqref{interm19bis}, we infer that also
$$
0<\sigma<\frac{N\beta^2}{4\alpha^2X(P_2)}
$$
and thus $X(Z-1)<k_1^2$, which implies that the orbit we are analyzing will enter the region $\{Y<0\}$ (and thus remain forever) below the surface \eqref{interm19} due to its negative direction of the flow.

\medskip

\noindent \textbf{Step 3. Barrier by vertical plane}. Finally, considering the plane $Y=-k_1$, the direction of the flow on this plane is given by the sign of the expression
$$
L(X,Z)=-X(Z-1)+\frac{\beta^2}{4\alpha^2}+\frac{N\beta}{2\alpha}X>0,
$$
in the region lying below the surface \eqref{interm19}, where our orbits lie for $\sigma$ as in \eqref{interm19bis}. This gives that the orbits going out of $P_2$ and $P_0$ cannot cross the plane $Y=-\beta/2\alpha$ and thus cannot reach the critical point $P_1$.
\end{proof}
The main problem comes with noticing that, for $p>p_c(\sigma)$, the plane \eqref{plane.roof} is no longer a good barrier for the flow, as then $Z_2<Z_1$ and orbits can escape in the interval $(Z_2,Z_1)$ while still in the half-space $\{Y>0\}$. This is why, we need a different construction, which comes from a surface that was very successful in the limiting case $p=m$. Let us perform first the change of variable $W=XZ$ in \eqref{PSsyst1} to obtain the following system we will work with in the rest of this section
\begin{equation}\label{PSsyst2}
\left\{\begin{array}{ll}\dot{X}=X[(m-1)Y-2X],\\
\dot{Y}=-Y^2-\frac{\beta}{\alpha}Y+X-NXY-W,\\
\dot{W}=W[(m+p-2)Y+(\sigma-2)X].\end{array}\right.
\end{equation}
We define in these variables the following surface, similar to the one that has been used as a barrier for the flow in the case $p=m$ in \cite{IS22}:
\begin{equation}\label{surface}
W=\left(-N-\frac{\sigma}{2}+1\right)XY-mY^2-\frac{(2N+\sigma-2)(\sigma+2)}{8m}X^2+\frac{2m}{m+1}X
\end{equation}
with normal vector
$$
\overline{n}(X,Y,W)=\left(\frac{\partial W}{\partial X},\frac{\partial W}{\partial Y},-1\right),
$$
where the partial derivatives correspond to the right hand side of the expression \eqref{surface}. The reader can see a picture of this surface in Figure \ref{fig2}. The main technical step in the proof of Theorem \ref{th.class}, part 1, is the following
\begin{lemma}\label{lem.surface}
Let $N\geq2$. Then for any $-2<\sigma\leq0$ and for any $p\in(p_c(\sigma),m)$, the direction of the flow of the system \eqref{PSsyst2} through the portion of the surface \eqref{surface} with $X<X(P_2)$ is in the decreasing direction with respect to the $W$ coordinate.
\end{lemma}
Some of the calculations in the following proof of Lemma \ref{lem.surface} have been performed with the aid of a symbolic computing software.
\begin{proof}
Somehow tedious but direct calculations give that the direction of the flow of the system \eqref{PSsyst2} over the surface \eqref{surface} is given by the following seven-term expression
\begin{equation}\label{flow}
\begin{split}
F(X,Y)&=-m(m-p)Y^3-\frac{1}{2}(m-p)(2N+\sigma-2)XY^2+\frac{2m(m-p)}{\sigma+2}Y^2\\
&-\frac{(m-p)(\sigma+2)(2N+\sigma-2)}{8m}X^2Y+\frac{(m-p)(2mN+2N+5m\sigma+6m+\sigma-2)}{2(\sigma+2)(m+1)}XY\\
&-\frac{(2N+\sigma-2)(\sigma+2)(2N-\sigma-6)}{16m}X^3+\frac{2(N-1)(m-1)-(3m+1)\sigma}{2(m+1)}X^2.
\end{split}
\end{equation}
The rest of the proof will consist of a careful compensation between terms in \eqref{flow} in order to show that $F(X,Y)>0$ for any $(X,Y)$ such that $0\leq X<X(P_2)$ and $Y<Y(P_2)$, of course with $\sigma$ and $p$ as in the statement of Lemma \ref{lem.surface}. The analysis will be also split on the regions $\{Y>0\}$ and $\{Y<0\}$ of the surface we consider, the latter being more involved. We divide it into a number of technical steps for the easiness of the reading.

\medskip

\noindent \textbf{Step 1. Terms in $X^2$ and $X^3$, taking a half of the term in $X^2$}. We deal with the last two terms in \eqref{flow}. In fact, optimizing a bit, we will show that it is enough to take only a half of the term in $X^2$. By dividing by $X^2$, the sign of their combination is the same as the sign of the expression
$$
E_1(X)=-\frac{(2N+\sigma-2)(\sigma+2)(2N-\sigma-6)}{16m}X+\frac{2(N-1)(m-1)-(3m+1)\sigma}{4(m+1)}.
$$
Notice that for $\sigma\in(-2,0]$ the free term in $E_1(X)$ (which is the same as $E_1(0)$) is positive, and it is then sufficient to show that $E_1(X(P_2))>0$ to complete the proof, since $E_1(X)$ is a linear function. We then have
$$
E_1(X(P_2))=-\frac{L_1(m,p,N,\sigma)}{32m(m+1)(mN-N+2)},
$$
where $L_1(m,p,N,\sigma)$ is an expression that can be written as a third degree polynomial with respect to $\sigma$ as follows
\begin{equation*}
\begin{split}
L_1&(m,p,N,\sigma)=-(m+1)(m-1)^2\sigma^3-2(m-1)(m+1)(2m+p-3)\sigma^2\\
&+\big[4(m+1)(m-1)^2N^2+8(m-1)(m^2+m+2)N+12m^3-8m^2p+44m^2+4m+8p+4\big]\sigma\\
&-8(m-1)(N-1)[((m+1)(m-p)+(m-1)^2)N+3mp+m+3(p-1)].
\end{split}
\end{equation*}
We notice that for $\sigma\in(-2,0)$, all the summands are negative except for the first one. But it is very easy to see that we obtain a negative result by coupling the first two terms in $L_1(m,p,N,\sigma)$, namely
$$
-(m+1)(m-1)\sigma^2[(m-1)(\sigma+2)+2(m+p-2)]<0.
$$
We then infer that $E_1(X(P_2))>0$, as desired.

\medskip

\noindent \textbf{Step 2. Region $\{Y>0\}$, terms in $X^2Y$ and $XY$}. The idea is the same as in the previous step, noticing that we can divide by $XY>0$ and the sign of this combination of terms in \eqref{flow} is the same as the sign of the following expression
$$
E_2(X)=-\frac{(m-p)(\sigma+2)(2N+\sigma-2)}{8m}X+\frac{(m-p)(2mN+2N+5m\sigma+6m+\sigma-2)}{2(\sigma+2)(m+1)}
$$
This is again a linear function in $X$ with $E_2(0)>0$ in an obvious way, thus it suffices to prove that also $E(X(P_2))\geq0$. We have
$$
E(X(P_2))=\frac{(m-p)L_2(m,p,N,\sigma)}{16m(mN-N+2)(m+1)(\sigma+2)},
$$
where $L_2(m,p,N,\sigma)$ is linear with respect to $p$. It is then enough to show that it is positive for $p=m$ and $p=1$. In the former case, we have
\begin{equation*}
\begin{split}
L_2&(m,m,N,\sigma)=-(m+1)(m-1)^2\sigma^3-(2(m+1)(m-1)^2N+2(m+1)(m-1)^2)\sigma^2\\
&+(8(m-1)(4m^2+m+1)N+4m^3+76m^2+12m+4)\sigma\\&+16m(m-1)(m+1)N^2+(40m^3-24m^2+56m-8)N+8m^3+88m^2-40m+8,
\end{split}
\end{equation*}
which is a cubic polynomial with respect to $\sigma$. We readily find that at $\sigma=0$ it gives
$$
L_2(m,m,N,0)=16m(m-1)(m+1)N^2+(40m^3-24m^2+56m-8)N+8m^3+88m^2-40m+8>0,
$$
while its evaluation at $\sigma=-2$ gives
$$
L_2(m,m,N,-2)=16m(m+1)(N-2)(Nm-N+2)\geq0,
$$
for $N\geq2$. Moreover, its second derivative with respect to $\sigma$ has the following expression
$$
-2(m+1)(m-1)^2[2(N+\sigma)+2+\sigma]<0, \qquad {\rm for} \ \sigma\in(-2,0),
$$
thus by basic calculus arguments we infer that $L_2(m,m,N,\sigma)\geq0$ for any $\sigma\in(-2,0]$. A similar analysis holds true for $p=1$, where
\begin{equation*}
\begin{split}
L_2(m,1,N,\sigma)&=-(m+1)(m-1)^2\sigma^3-2(m+1)(m-1)^2N\sigma^2\\&+(4(m-1)(9m^2+2m+1)N+4m^3+76m^2+12m+4)\sigma\\&+16m(m-1)(m+1)N^2+16m(3m^2-2m+3)N+32m(3m-1).
\end{split}
\end{equation*}
We then have
$$
L_2(m,1,N,0)=16m(m-1)(m+1)N^2+16m(3m^2-2m+3)N+32m(3m-1)>0,
$$
while
$$
L_2(m,1,N,-2)=16m(m+1)(N-2)(Nm-N+2)\geq0, \qquad {\rm for} \ N\geq2.
$$
Moreover, the second derivative of the polynomial $L_2(m,1,N,\sigma)$ with respect to $\sigma$ is
$$
-2(m+1)(m-1)^2(3\sigma+2N)
$$
which is obviously negative for $N\geq3$ in our range of $\sigma$. Finally, for the case $N=2$ we consider the first derivative with respect to $\sigma$ of $L_2(m,1,2,\sigma)$, which is a second degree polynomial in $\sigma$ with dominating negative coefficient, and observe that
$$
\frac{\partial L_2}{\partial\sigma}\Big|_{\sigma=-2}=16m^2(5m+1)>0, \ \frac{\partial L_2}{\partial\sigma}\Big|_{\sigma=0}=76m^3+20m^2+4m-4>0,
$$
whence $L_2(m,1,2,\sigma)$ is increasing for $\sigma\in(-2,0]$ for $N=2$, which leads to its positivity.

\medskip

\noindent \textbf{Step 3. Region $\{Y>0\}$, terms in $Y^3$, $Y^2$ and $XY^2$}. We are left with the sign of the combination of the first three terms in \eqref{flow}, which after dividing by $Y^2$, is the same as the sign of the expression
$$
E_3(X,Y)=-\frac{m-p}{2(\sigma+2)}\left[(2N+\sigma-2)(\sigma+2)X+2m(\sigma+2)Y-4m\right]
$$
and since the term in brackets is linearly increasing in both $X$ and $Y$, it is sufficient to show that $E_3(X(P_2),Y(P_2))>0$. To this end, we find that
$$
E_3(X(P_2),Y(P_2))=-\frac{m-p}{2(\sigma+2)}\frac{[2(m-1)^2\sigma-4(m-1)(2m-p+1)]N+L_3(m,p,\sigma)}{2(mN-N+2)},
$$
where
$$
L_3(m,p,\sigma)=(m-1)^2\sigma^2+2(m-1)(m+p)\sigma+4mp-20m+4p-4.
$$
Notice that the numerator of $E_3(X(P_2),Y(P_2))$ is a linearly decreasing expression of $N$, thus it suffices to prove that this numerator is negative for $N=1$. And this is achieved by standard calculus tools, noticing that the expression giving the sign can be written, for $N=1$, as
$$
P(\sigma)=(m-1)^2\sigma^2+2(m-1)(2m-1+p)\sigma-8m(m-p+2),
$$
which is a second degree polynomial with dominating positive coefficient and such that
$$
P(0)=-8m(m-p+2)<0, \qquad P(-2)=-4(3m-p)(m+1)<0,
$$
thus $P(\sigma)<0$ for any $\sigma\in(-2,0)$. It then follows that $E_3(X(P_2),Y(P_2))>0$ and thus $E_3(X,Y)>0$ in the region we are interested in. The analysis of the case $Y>0$ is completed, since the remaining half of the term in $X^2$ in \eqref{flow} is obviously positive for $\sigma\in(-2,0]$.

\medskip

\noindent \textbf{Step 4. Compensation of the term in $X^3$ in \eqref{flow} by one third of the term in $X^2$, for $m\geq3$}. This is a technical improvement over Step 1 above (where we were taking a half of the term in $X^2$) needed in order to deal with the region of the surface \eqref{surface} lying in the half-space $\{Y<0\}$, but only works for $m\geq3$. We proceed as in Step 1, and deduce that the sign of this combination of terms is given by the following expression
$$
\tilde{E}_1(X)=-\frac{(2N+\sigma-2)(\sigma+2)(2N-\sigma-6)}{16m}X+\frac{2(N-1)(m-1)-(3m+1)\sigma}{6(m+1)},
$$
and once more we wish to show that $\tilde{E}_1(X(P_2))>0$. We then have
$$
\tilde{E}_1(X(P_2))=-\frac{\tilde{L}_1(m,p,N,\sigma)}{96m(mN-N+2)(m+1)},
$$
where $\tilde{L}_1(m,p,N,\sigma)$ can be written as a third degree polynomial in $\sigma$ as follows:
\begin{equation*}
\begin{split}
\tilde{L}_1&(m,p,N,\sigma)=-3(m+1)(m-1)^2\sigma^3-6(m-1)(m+1)(2m+p-3)\sigma^2\\
&+\big[12(m+1)(m-1)^2N^2+16(m+3)(m-1)N+36m^3-24m^2p+84m^2-4m+24p+12\big]\sigma\\
&-8(m-1)(N-1)[(4m^2-(3p+1)m-3p+3)N+(9p-1)m+9(p-1)].
\end{split}
\end{equation*}
We easily observe that, for $\sigma=0$, we are left with the free term in the expression of $\tilde{L}_1$ and this is a linear function of $p$. Letting $p=m$, respectively $p=1$, we get
$$
\frac{\tilde{L}_1(m,m,N,0)}{8(m-1)(N-1)}=-(m-1)(m-3)N-9m^2-8m+9, \qquad \frac{\tilde{L}_1(m,1,N,0)}{8(m-1)(N-1)}=-4m(m-1)N-8m
$$
and both are negative provided that $m\geq3$. By linearity, we infer that $\tilde{L}_1(m,p,N,0)<0$. On the other hand, similar calculations for $\sigma=-2$, respectively $p=m$ and $p=1$ give
$$
\tilde{L}_1(m,m,N,-2)=-32m(Nm-N+2m+2)(Nm-N+2)<0
$$
and
$$
\tilde{L}_1(m,1,N,-2)=-8(7m+3)(m-1)^2N^2+32(m-1)(m^2-4m-3)N-32(m+1)(3m^2-2m+3)
$$
which is straightforward to see that it is negative for $N>1$. We again infer by linearity that $\tilde{L}_1(m,p,N,-2)<0$. The final argument is to see that $\tilde{L}_1(m,p,N,\sigma)$ is monotone increasing with respect to $\sigma$ for $\sigma\in(-2,0)$. Indeed, its derivative with respect to $\sigma$
\begin{equation*}
\begin{split}
\frac{\partial \tilde{L}_1}{\partial\sigma}&=-9(m+1)(m-1)^2\sigma^2-12(m-1)(m+1)(2m+p-3)\sigma\\
&+12(m+1)(m-1)^2N^2+16(m+3)(m-1)N+36m^3-24m^2p+84m^2-4m+24p+12
\end{split}
\end{equation*}
is a second degree polynomial in $\sigma$ with negative dominating coefficient and such that it is positive at $\sigma=0$ and $\sigma=-2$. The former is obvious by examining the free term, while the latter follows from noticing that
$$
\frac{\partial \tilde{L}_1}{\partial\sigma}(m,p,N,-2)=12(m+1)(m-1)^2N^2+16(m+3)(m-1)N+48m^3+48m^2-16m+48
$$
which is obviously positive.

\medskip

\noindent \textbf{Step 5. Region $\{Y<0\}$, compensation of terms in $XY^2$, $Y^3$ and one half of the term in $Y^2$ in \eqref{flow}}. We take the full terms in $Y^3$ and $XY^2$ in \eqref{flow} and only a half of the term in $Y^2$ in \eqref{flow}. After dividing by $Y^2$, we notice that the sign of the combination of these terms is the same as the sign of
$$
E_4(X,Y)=-\frac{(m-p)[(2N+\sigma-2)(\sigma+2)X+2m(\sigma+2)Y-2m]}{2(\sigma+2)}.
$$
Since now $Y<0$, it suffices to show that
$$
L_4(m,p,N,\sigma)=(2N+\sigma-2)(\sigma+2)X(P_2)-2m<0
$$
for any $\sigma\in(-2,0)$. We then have
\begin{equation*}
\begin{split}
L_4(m,p,N,\sigma)&=\frac{1}{2(mN-N+2)}\left[(m-1)^2\sigma^2+[2(m-1)^2N-2(m-1)(m-p)]\sigma\right.\\
&\left.-4(m-1)(m-p+1)N-4mp-4m+4p-4\right]
\end{split}
\end{equation*}
which is again a second degree polynomial in $\sigma$ with positive dominating coefficient. Noticing that at the endpoints one gets
\begin{equation*}
\begin{split}
&L_4(m,p,N,0)=-\frac{2(m-1)(m-p+1)N+2mp+2m-2p+2}{mN-N+2}<0, \\ &L_4(m,p,N,-2)=-\frac{2(m-1)(2m-p)N-4m^2+4mp+8m-4p}{mN-N+2},
\end{split}
\end{equation*}
and that the numerator of $L_4(m,p,N,-2)$ is a linear expression in $N$ with negative coefficient such that at $N=1$ we have
$$
L_4(m,p,1,-2)=-\frac{2mp+4m-2p}{m+1}<0,
$$
we infer that $L_4(m,p,N,\sigma)<0$ for any $\sigma\in(-2,0]$ and thus $E_4(X,Y)>0$.

\medskip

\noindent \textbf{Step 6. Region $\{Y<0\}$, compensation of terms in $XY^2$, $Y^3$ and one third of the term in $Y^2$ in \eqref{flow} if $m\leq3$}. In the same way as in the previous step, the sign of this combination of terms is the same as the sign of the expression
$$
\tilde{E}_4(X,Y)=-\frac{(m-p)[3(2N+\sigma-2)(\sigma+2)X+6m(\sigma+2)Y-4m]}{6(\sigma+2)}
$$
and we are interested to prove that the expression
$$
\tilde{L}_4(m,p,N,\sigma)=3(2N+\sigma-2)(\sigma+2)X(P_2)-4m<0
$$
for any $\sigma\in(-2,0)$, provided that $m\leq3$. We have
\begin{equation*}
\begin{split}
\tilde{L}_4(m,p,N,\sigma)&=\frac{1}{2(mN-N+2)}\left[3(m-1)^2\sigma^2+[6(m-1)^2N-6(m-1)(m-p)]\sigma\right.\\
&\left.-4(m-1)(2m-3p+3)N-12mp-4m+12p-12\right],
\end{split}
\end{equation*}
which is again a second degree polynomial in $\sigma$ with positive dominating coefficient. At the endpoint $\sigma=-2$ we notice that
$$
\tilde{L}_4(m,p,N,-2)=-\frac{4(m-1)(5m-3p)N-24m^2+24mp+40m-24p}{2(mN-N+2)}
$$
and this is negative for any $N\geq2$, since its numerator is a linear expression with respect to $N$, with positive coefficient, and at $N=2$ we readily have $\tilde{L}_4(m,p,2,-2)=-4m<0$. Things are a bit more complicated at the endpoint $\sigma=0$, since
$$
\tilde{L}_4(m,p,N,0)=-\frac{4(m-1)(2m-3p+3)N+12mp+4m-12p+12}{2(mN-N+2)}
$$
whose coefficient as a function of $N$ might change sign. But we notice that its numerator is a linear function of $p$ and that at $p=m$ and at $p=1$
$$
\tilde{L}_4(m,m,N,0)=\frac{4(m-1)(m-3)N-12m^2+8m-12}{2(mN-N+2)}, \ \tilde{L}_4(m,1,N,0)=\frac{-8(m-1)mN-16m}{2(mN-N+2)}
$$
both expressions have negative signs. It then follows that $\tilde{L}_4<0$ for any $\sigma\in(-2,0]$ and thus $\tilde{E}_4(X,Y)>0$ for any $m\in[1,3]$.

\medskip

\noindent \textbf{Step 7. Compensation of the terms in $X^2$, $Y^2$ and $XY$ in \eqref{flow}. End of the proof for $N\geq5$}. From previous steps, we were left with some parts of the terms in $X^2$ and $Y^2$, as either one half or one third of them (according to each step) have been used in different other balances of terms. We are thus interested in showing that an expression of the form
\begin{equation}\label{interm22}
\begin{split}
a&\frac{2m(m-p)}{\sigma+2}Y^2+\frac{(m-p)(2Nm+5m\sigma+2N+6m+\sigma-2)}{2(\sigma+2)(m+1)}XY\\&+b\frac{2(N-1)(m-1)-(3m+1)\sigma}{2(m+1)}X^2\\
&=\frac{(m-p)[(2Nm+5m\sigma+2N+6m+\sigma-2)X+8(m+1)amY]^2}{32am(m+1)^2(\sigma+2)}\\
&+b\left[\frac{2(N-1)(m-1)-(3m+1)\sigma}{2(m+1)}X^2-\frac{(m-p)(2Nm+5m\sigma+2N+6m+\sigma-2)^2}{32(\sigma+2)(m+1)^2abm}X^2\right]
\end{split}
\end{equation}
with $a$ and $b$ to be fixed according to each case, has positive sign. Let us notice that the first term in the right hand side of the equality \eqref{interm22} is a square, thus it is sufficient to prove that the terms in the final line of \eqref{interm22} give a positive contribution. Noticing that this sign only depends on the value of the product $ab$, we begin with $a=1/2$ and $b=2/3$ and find that the combination of the terms in the last line of \eqref{interm22} gives
$$
\frac{X^2}{48m(m+1)^2(\sigma+2)}L_5(m,p,N,\sigma),
$$
where $L_5(m,p,N,\sigma)$ is a linear term with respect to $p$ (whose expression we omit here for simplicity). It is then sufficient to show that it is negative at $p=m$ and at $p=p_c(\sigma)$ in order to conclude that it has negative sign in the middle, where we recall that $p_c(\sigma)$ has been introduced in \eqref{branch}. We indeed have
$$
L_5(m,m,N,\sigma)=16[2(N-1)(m-1)-(3m+1)\sigma](\sigma+2)(m+1)m>0.
$$
It is right now where the lower limit $p=p_c(\sigma)$ already established in Lemma \ref{lem.rooftop} comes into action. Indeed, evaluating $L_5$ at $p=1$ does not lead to the desired positive sign, but instead, since we already covered the interval $p\in[1,p_c(\sigma)]$, we can work starting from $p=p_c(\sigma)$. We thus have
$$
L_5(m,p_c(\sigma),N,\sigma)=\frac{P(m,N,\sigma)(\sigma+2)}{N+\sigma+2},
$$
where $P(m,N,\sigma)$ is a polynomial of second degree in $\sigma$ as follows
\begin{equation*}
\begin{split}
P(m,N,\sigma)&=-(3m-1)(41m^2+20m+3)\sigma^2\\&-[4(m+1)(19m^2-3)N+308m^3-76m^2+12m+12]\sigma\\
&+4(m-1)[(5m-3)(m+1)N^2-2(5m-3)(m+1)N-43m^2+2m-3].
\end{split}
\end{equation*}
Notice that $P(m,N,\sigma)$ has negative dominating coefficient and by analyzing $P(m,N,0)$ which is the free term (as a polynomial with respect to $N$) we easily notice that it is positive for any $N\geq5$ (but unfortunately not for $N=2,3,4$, cases that have to be considered separately in forthcoming steps). Moreover,
$$
P(m,N,-2)=4(m+1)[(5m-3)(m-1)N^2+4(m+1)(7m-3)N-12(m-1)(m+1)]
$$
is obviously positive for $N\geq5$. Thus $P(m,N,\sigma)>0$ for any $\sigma\in(-2,0]$, which implies also the positivity of the combination of terms we started with in this step. Thus, for $N\geq5$, if $m\geq3$ we complete the proof of the Lemma by using the fractions of terms as in Step 4, Step 5, together with $a=1/2$ and $b=2/3$ in the current step, while if $m\leq3$ we complete the proof by using the fractions of terms as in Step 1, Step 6 together with $a=2/3$ and $b=1/2$ in the current step.

\medskip

\noindent \textbf{Step 8. Dimensions $N=3,4$. Compensation of the terms in $X^2$, $Y^2$ and $XY$ in \eqref{flow}}. We are left with dimensions $N=3$ and $N=4$ where Step 7 above did not work with the coefficients $a=1/2$ and $b=2/3$. We thus work in a completely similar manner but letting $a=1/2$ and $b=9/10$ in \eqref{interm22}. In this case, the relevant term for the sign (given by the combination of terms in the last line of \eqref{interm22}) gives
$$
-\frac{Q(m,p,N,\sigma)}{80m(m+1)^2(\sigma+2)}X^2,
$$
where $Q(m,p,N,\sigma)$ is a linear function with respect to $p$ and at the same time a second degree polynomial in $\sigma$. On the one hand, for $p=m$ and $\sigma\in(-2,0]$, we compute
$$
Q(m,m,N,\sigma)=-36m(m+1)(\sigma+2)[2(N-1)(m-1)-(3m+1)\sigma]<0.
$$
On the other hand, if $p=p_c(\sigma)$, we get
$$
Q(m,p_c(\sigma),N,\sigma)=-\frac{\sigma+2}{N+\sigma+2}R(m,N,\sigma),
$$
where $R(m,N,\sigma)$ is a second degree polynomial in $\sigma$, more precisely
\begin{equation*}
\begin{split}
R(m,N,\sigma)&=-(233m^3+69m^2-9m-5)\sigma^2\\
&-(136Nm^3+164Nm^2+588m^3+8Nm-52m^2-20N+20m+20)\sigma\\
&+4(m-1)[(m+1)(13m-5)N^2-2(m+1)(6m-5)N-81m^2-6m-5].
\end{split}
\end{equation*}
It is easy to see that $R(m,N,0)$ (that is, the free term) is positive for $N\geq3$. Moreover,
$$
R(m,N,-2)=4(m+1)[(13m-5)(m-1)N^2+4(m+1)(14m-5)N-20(m-1)(m+1)]>0
$$
provided $N\geq3$. It then follows that $R(m,N,\sigma)>0$ and thus $Q(m,p,N,\sigma)<0$ for any $\sigma\in(-2,0]$, as desired.

\medskip

\noindent \textbf{Step 9. Dimensions $N=3,4$. Compensation of the term in $X^3$ with one tenth of the term in $X^2$ in \eqref{flow}}. Since we had to use $b=9/10$ in front of the term involving $X^2$ in Step 8, we are only left with $1/10$ of the term involving $X^2$ in \eqref{flow} in order to compensate the term in $X^3$, thus we need to improve Step 1 for dimensions $N=3$ and $N=4$. We again divide by $X^2$ and then notice that it is sufficient to prove that the remaining term is positive evaluated at $X=X(P_2)$, as we did in Step 1. Proceeding exactly the same as in Step 1, we obtain that the term whose sign we have to study is given by
\begin{equation*}
\begin{split}
P_3&(m,p,\sigma)=5(m+1)(m-1)^2\sigma^3+10(m-1)(m+1)(2m+p-3)\sigma^2\\
&+(-72m^3+40m^2p-40m^2+8m-40p+40)\sigma+32m(3m-1)(m-1),
\end{split}
\end{equation*}
for $N=3$ and
\begin{equation*}
\begin{split}
P_4&(m,p,\sigma)=5(m+1)(m-1)^2\sigma^3+10(m-1)(m+1)(2m+p-3)\sigma^2\\
&+(-156m^3+40m^2p+36m^2+76m-40p-20)\sigma+24(m-1)(8m^2-5p(m+1)+m+5),
\end{split}
\end{equation*}
for $N=4$. Since for $p\in[1,m]$ we have
$$
8m^2-5p(m-1)+m+5\geq8m^2-5m(m-1)+m+5=3m^2-4m+5>0,
$$
we notice that all the summands but the first one are positive for $\sigma\in(-2,0]$ in both polynomials $P_3$ and $P_4$, but the first term is dominated by the second term since
$$
5(m+1)(m-1)^2\sigma^3+10(m-1)(m+1)(2m+p-3)\sigma^2=5(m-1)(m+1)\sigma^2[(\sigma+4)(m-1)+2(p-1)]
$$
is a positive contribution. The proof is now complete also for dimensions $N=3$ and $N=4$.

\medskip

\noindent \textbf{Step 10. Dimension $N=2$}. We are left with the case of dimension $N=2$, which is easier than the previous ones. Let us notice that on the one hand Steps 2 and 3 (regarding the region $\{Y>0\}$) also hold true for $N=2$, while on the other hand the term in $X^3$ in \eqref{flow} is now positive, since $2N-\sigma-6=-(\sigma+2)<0$ for $N=2$. We are thus left with the region $\{Y<0\}$, but as we shall see below, the argument now follows just by completing squares. Thus, the terms in $Y^3$, $XY^2$ and $X^2Y$ in \eqref{flow} give for $N=2$
$$
-\left[m(m-p)\left(Y+\frac{(\sigma+2)X}{4m}\right)^2+\frac{(\sigma+2)^2(m-p)X^2}{16m}\right]Y>0,
$$
since we are working on the region $\{Y<0\}$. We are left with the terms in $X^2$, $Y^2$ and $XY$ in \eqref{flow}. Their joint contribution gives after completing squares the following term
\begin{equation}\label{interm23}
\begin{split}
H(X,Y)&=\frac{2m(m-p)}{\sigma+2}\left[Y+\frac{5m\sigma+10m+\sigma+2}{8m(m+1)}X\right]^2\\
&+\left[\frac{2(m-1)-(3m+1)\sigma}{2(m+1)}-\frac{(m-p)(5m+1)}{32m(m+1)^2}\right]X^2.
\end{split}
\end{equation}
It is then enough to show that the coefficient multiplying $X^2$ in \eqref{interm23} is positive in order to complete the proof. This is obvious at $p=m$ where the only negative term is canceled. Since this coefficient is a linear function of $p$, it suffices to prove its positivity also at $p=p_c(\sigma)$, where its expression is
$$
\frac{2(m-1)-(3m+1)\sigma}{2(m+1)}-\frac{(m-p_c(\sigma))(5m+1)}{32m(m+1)^2}=-\frac{R(m,\sigma)}{32m(m+1)^2(\sigma+4)},
$$
where
$$
R(m,\sigma)=(73m^3+49m^2+7m-1)\sigma^2+(260m^3+196m^2+60m-4)\sigma-4(m-1)(7m^2+22m-1).
$$
We readily observe that
$$
R(m,0)=-4(m-1)(7m^2+22m-1)<0, \qquad R(m,-2)=-256m^2(m+1)<0,
$$
hence $R(m,\sigma)<0$ for any $\sigma\in(-2,0]$ as a second degree polynomial with positive dominating coefficient. It thus follows that $H(X,Y)>0$ and the proof is complete.
\end{proof}
We infer from Lemma \ref{lem.walls} and Lemma \ref{lem.surface} that , for $\sigma\in(-2,0]$ and $p\in(p_c(\sigma),m)$, the region $\mathcal{W}$ limited by the planes $X=X(P_2)$, $Y=Y(P_2)$, $W=0$ and the surface \eqref{surface} is positively invariant for the flow of the system \eqref{PSsyst2}: once an orbit enters it, it cannot go out afterwards. We then analyze the orbits going out of $P_2$ and $P_0$.
\begin{lemma}\label{lem.roofsurf}
Let $N\geq2$, $\sigma\in(-2,0]$ and $p_c(\sigma)\leq p<m$. Then the orbits going out of $P_2$ and $P_0$ in the phase space associated to the system \eqref{PSsyst2} enter the region $\mathcal{W}$.
\end{lemma}
\begin{proof}
We prove first that the critical point $P_2$ lies itself ``below" the surface \eqref{surface}. To this end, it is enough to evaluate \eqref{surface} at $(X,Y)=(X(P_2),Y(P_2))$ and show that $W=W(X(P_2),Y(P_2))>0=W(P_2)$. We thus have at $(X,Y)=(X(P_2),Y(P_2))$
$$
W=-\frac{S(m,p,N,\sigma)}{32m(m+1)(mN-N+2)^2(\sigma+2)^2},
$$
where $S(m,p,N,\sigma)$ is a rather long expression (whose exact form we omit here for simplicity) which is linear in $p$. Moreover, we obtain at $p=m$
\begin{equation*}
\begin{split}
S(m,m,N,\sigma)&=(m-1)^2(\sigma+2)\left[(m-1)(m+1)\sigma^2\right.\\&\left.+2(m+1)((m-1)N+4m)\sigma-4(m-1)(3m-1)(N-1)\right].
\end{split}
\end{equation*}
and the factor in brackets is negative both at $\sigma=0$ and at $\sigma=-2$, where it gives $-16m(mN-N+2)<0$. Since this factor is a second degree polynomial in $\sigma$ with positive dominating coefficient, we infer that $S(m,m,N,\sigma)<0$ for any $\sigma\in(-2,0]$. Evaluating now at $p=1$, we get
\begin{equation*}
\begin{split}
S(m,1,N,\sigma)&=(m+1)(m-1)^3\sigma^3+2(m-1)^2(m+1)[(m-1)N+4m]\sigma^2\\
&-4(m-1)(3m-1)(N-1)\sigma-64(m-1)m^2(mN-N+2),
\end{split}
\end{equation*}
which is a cubic polynomial in $\sigma$ such that the values at the endpoints $\sigma=0$ and $\sigma=-2$
$$
S(m,1,N,0)=-64(m-1)m^2(mN-N+2), \ S(m,1,N,-2)=-32m(m-1)(m+1)(mN-N+2),
$$
are both negative and its second derivative is positive for any $\sigma\in[-2,0]$ (we leave to the reader the verification of this rather easy fact). By standard calculus results it follows that $S(m,1,N,\sigma)<0$ and thus $S(m,p,N,\sigma)<0$ for any $\sigma\in(-2,0]$ and $p\in(p_c(\sigma),m)$. This proves that the point $P_2$ lies below the surface \eqref{surface} and thus Lemmas \ref{lem.P2} and \ref{lem.walls} imply that the orbit from $P_2$ goes out into $\mathcal{W}$.

We draw now our attention to the orbits going our of $P_0$. We recall from Lemma \ref{lem.P0} that these orbits go out tangent to the center manifold \eqref{center.P0}, which in variables $(X,Y,W)$ writes
\begin{equation}\label{center.P00}
W_0=X-\frac{\beta}{\alpha}Y-\frac{\alpha}{\beta^2}[(N-2)\beta+m\alpha]X^2.
\end{equation}
We then notice that
\begin{equation*}
\begin{split}
&W-W_0=\frac{m-1}{m+1}X+\frac{m-p}{\sigma+2}Y-\frac{2N+\sigma-2}{2}XY-mY^2+\frac{\sigma+2}{8m(m-p)^2}A(m,p,N,\sigma)X^2,\\
&A(m,p,N,\sigma)=(7m^2+2mp-p^2)\sigma+2(3m+p)(m-p)N+2m^2+12mp+2p^2,
\end{split}
\end{equation*}
hence in a small neighborhood of the critical point $P_0$ we find $W-W_0>0$, since the orbits going out of $P_0$ enter the positive half-space $\{Y>0\}$. This shows that the orbits go out directly into $\mathcal{W}$.
\end{proof}
We represent in the Figure \ref{fig2} the surface defined in \eqref{surface} and the center manifold given in \eqref{center.P00}, and we see that for $X$ sufficiently small, the center manifold lies in the region $\mathcal{W}$, as proved. 

\begin{figure}[ht!]
  \begin{center}
  \includegraphics[width=11cm,height=7.5cm]{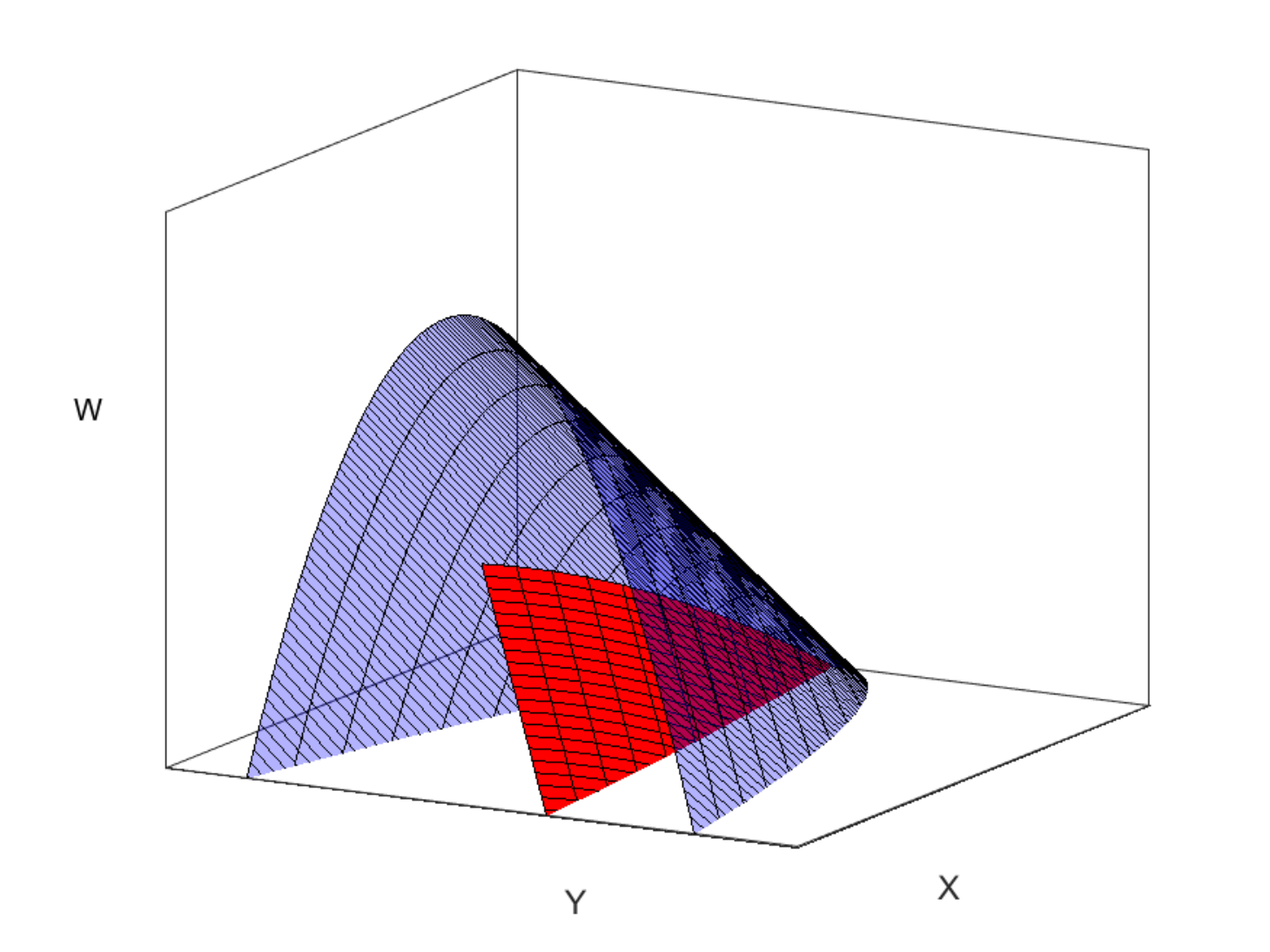}
  \end{center}
  \caption{The surface defined in \eqref{surface} and the center manifold \eqref{center.P00}.}\label{fig2}
\end{figure}

\medskip

\noindent \textbf{Differences in dimension $N=1$}. We are left with dimension $N=1$, where the previous analysis does not work. In change, we introduce a \emph{different surface} for this special case:
\begin{equation}\label{surface1}
\begin{split}
W&=-\frac{(m+p)\sigma(2m\sigma+3m+p)}{2m(3m+p)^2}X^2-\frac{(m+p)\sigma}{3m+p}XY-\frac{m+p}{2}Y^2\\
&+\frac{2(m+p)[(m+p)\sigma+3m+p]}{(m+1)(3m+p)(\sigma+2)}X.
\end{split}
\end{equation}
The following result shows that the surface \eqref{surface1} does the desired job for $N=1$.
\begin{lemma}\label{lem.roofsurfN1}
Let $N=1$. Then for any $\sigma\in(-2,0]$ and any $p\in[1,m)$, the region $\mathcal{V}$ limited by the planes $X=X(P_2)$, $Y=Y(P_2)$, $W=0$ and the surface \eqref{surface1} is positively invariant for the flow of the system \eqref{PSsyst2} and the orbits going out of $P_0$ and $P_2$ enter this region.
\end{lemma}
\begin{proof}[Sketch of the proof]
The flow of the system \eqref{PSsyst1} on the surface \eqref{surface1} is given by the sign of the expression
\begin{equation}\label{flow1}
\begin{split}
&F_1(X,Y)=\frac{(m+p)(m-p)}{\sigma+2}Y^2+\frac{(m+p)\sigma}{(3m+p)^2}X^2\\
&\times\left[\frac{(2m\sigma+3m+p)(m\sigma+3m+p)}{m(3m+p)}X-\frac{(7m^2+5mp+3m+p)\sigma+2(3m+1)(3m+p)}{(m+1)(\sigma+2)}\right]
\end{split}
\end{equation}
and in order to prove that the sign of $F_1(X,Y)$ is positive for $\sigma\in(-2,0)$ it is sufficient to check that the factor in brackets in the second line of \eqref{flow1} is negative, which is analyzed at $X=X(P_2)$ in a similar way as in Step 1 in the proof of Lemma \ref{lem.surface}. Moreover, rather similar calculations as in the proof of Lemma \ref{lem.roofsurf} give that $P_2$ lies ``below" the surface \eqref{surface1}, while the orbits going out of $P_0$ tangent to the center manifold \eqref{center.P0} do it also inside the region $\mathcal{V}$, since the coefficients of the linear terms in $X$ and $Y$ in the difference of surfaces are still positive for any $\sigma\in(-2,0)$. We omit here the technical details, since they follow the same arguments as in the proofs of Lemmas \ref{lem.surface} and \ref{lem.roofsurf}.
\end{proof}
We are finally in a position to complete the proof of Theorem \ref{th.class}, Part 1.
\begin{proof}[Proof of Theorem \ref{th.class}, Part 1]
We infer from Lemmas \ref{lem.walls}, \ref{lem.rooftop}, \ref{lem.surface},  and \ref{lem.roofsurf} in dimension $N\geq2$, respectively Lemma \ref{lem.roofsurfN1} in dimension $N=1$, that for any $p\in[1,m)$ and for any $\sigma\in(-2,0]$, the orbits going out of both $P_0$ and $P_2$ have to remain forever inside the positively invariant regions $\mathcal{W}$, respectively $\mathcal{V}$ and thus cannot reach the critical point $P_1$. Moreover, this property can be extended by standard continuity arguments up to some $\sigma_0>0$. We then infer from Theorem \ref{th.exist} that the only orbits entering $P_1$ and containing good profiles must necessarily come from the critical point $Q_1$ and we are done.
\end{proof}

\noindent \textbf{Remarks.} (a) In particular, we have given an alternative proof (to the one in \cite[Theorem 2, Section 1.3, Chapter 4]{S4}) of the existence of solutions stemming from the critical point $Q_1$ (in our notation) for the non-weighted case $\sigma=0$. Let us stress here that we also show that solutions with $f(0)>0$, $f'(0)=0$ are the \emph{only} possible solutions for $\sigma=0$.

\medskip

(b) It is very likely that the orbits going out of $P_0$ and $P_2$ for $\sigma\in(-2,\sigma_0)$ will actually enter the critical point $P_{\gamma}$. In order to prove rigorously this fact, we have to avoid the existence of limit cycles, which is not an easy problem in three-dimensional dynamical systems and when no coordinate among $X$, $Y$, $Z$ (or $W$) is monotone along the orbits we are looking at. We have dealt with such a situation in the recent work \cite[Section 2.3]{IMS22} and we believe that similar arguments might be used to avoid the existence of limit cycles, but we do not wish to enter this discussion here. We will go back to this question only for $\sigma>0$ (when some monotonicity is granted on the orbit going out of $P_2$) at the end of the next section.

\section{Classification of the profiles: $\sigma$ large}\label{sec.large}

In this final section, we complete the proof of Theorem \ref{th.class} by proving the remaining two items of it, dealing with values of $\sigma$ much larger than zero. Some parts of the proof might look a bit tedious and some of the calculations were performed with the aid of a symbolic calculation tool. We first need a preparatory result.
\begin{lemma}\label{lem.monot}
Let $\sigma>0$. Then the coordinate $X$ is decreasing and the coordinate $Y$ is also decreasing in the half-space $\{Y\geq0\}$ along the unique orbit going out of $P_2$ in the phase space associated to the system \eqref{PSsyst1}.
\end{lemma}
The proof is totally similar to the one of \cite[Lemma 5.3]{IS21a} where the same result is proved in the case $N=1$, or \cite[Lemma 5.1]{IMS22}. We are now in a position to proceed to the proof of the remaining items of Theorem \ref{th.class} and we begin with the last one of them.
\begin{proof}[Proof of Theorem \ref{th.class}, Part 3]
The proof is again divided into several steps for the reader's convenience. The scheme of the proof is based on constructing a region in the phase space, limited by two planes passing through $P_2$, such that any orbit entering this region has to go very far away in the sense of negative values of $Y$ in order to get out of it, and thus be forced to connect to $Q_3$. We then show that all the orbits stemming from $P_2$ and $Q_1$ for sufficiently large $\sigma$ cannot enter the critical point $P_1$, thus Theorem \ref{th.exist} implies that the good profiles contained in orbits entering $P_1$ must come from the remaining point $P_0$. Let us set $Y_0=(m-1)/2$.

\medskip

\noindent \textbf{Step 1. Plane of no return}. In this step, we show that, if an orbit of the system \eqref{PSsyst1} crosses the plane $\{Y=-Y_0\}$, then it has to enter the critical point $Q_3$. The direction of the flow of the system across the plane $\{Y=-Y_0\}$ is given by the sign of the expression
$$
R(X,Z)=\frac{mN-N+2}{2}X-\frac{(m-1)[\sigma(m-1)+2(p-1)]}{4(\sigma+2)}-XZ,
$$
and since $0\leq X\leq X(P_2)$ over the orbit going out of $P_2$, we observe that
$$
R(0,Z)=-\frac{(m-1)[\sigma(m-1)+2(p-1)]}{4(\sigma+2)}<0, \ R(X(P_2),Z)=-X(P_2)Z<0,
$$
hence $R(X,Z)<0$ for any $X\in[0,X(P_2)]$. It follows that the plane $\{Y=-Y_0\}$ cannot be crossed from the negative part, thus an orbit entering the region $\{Y<-Y_0\}$ will stay there forever. We then infer from the first and the second equation in the system \eqref{PSsyst1} that in the region $\{Y<-Y_0\}$ both components $X$ and $Y$ are monotone over any orbit, thus the orbit going out of $P_2$ cannot end up in a limit cycle. It has then to enter the stable node $Q_3$.

\medskip

\noindent \textbf{Step 2. Definitions of the two limiting planes}. We introduce the following planes
\begin{equation}\label{PlaneYZ}
DY+Z=E, \qquad D=\frac{2(mN-N+2)^2}{m-1}, \ E=\frac{2(mN-N+2)}{\alpha(m-1)}.
\end{equation}
and
\begin{equation}\label{PlaneXY}
\begin{split}
X=BY+C, \qquad &B=\frac{m(m-1)}{2(m-1)^2N^2+7(m-1)N+2(m+3)}, \\ &C=\frac{(2mN-2N+3)(m-1)[\sigma(m-1)+2(p-1)]}{2(\sigma+2)[2(m-1)^2N^2+7(m-1)N+2(m+3)]}.
\end{split}
\end{equation}
It is easy to check that both planes \eqref{PlaneYZ} and \eqref{PlaneXY} contain the critical point $P_2$. We show next that the following region limited by them
\begin{equation}\label{reg.large}
\mathcal{D}:=\{(X,Y,Z)\in\real^3: X>BY+C, Z>E-DY, 0\leq X\leq X(P_2), -Y_0\leq Y\leq Y(P_2)\}
\end{equation}
is almost invariant for the flow of the system \eqref{PSsyst1}, in the sense that it can be left only by crossing the plane $\{Y=-Y_0\}$. We plot the region $\mathcal{D}$ in Figure \ref{fig3}.

\begin{figure}[ht!]
  \begin{center}
  \includegraphics[width=11cm,height=7.5cm]{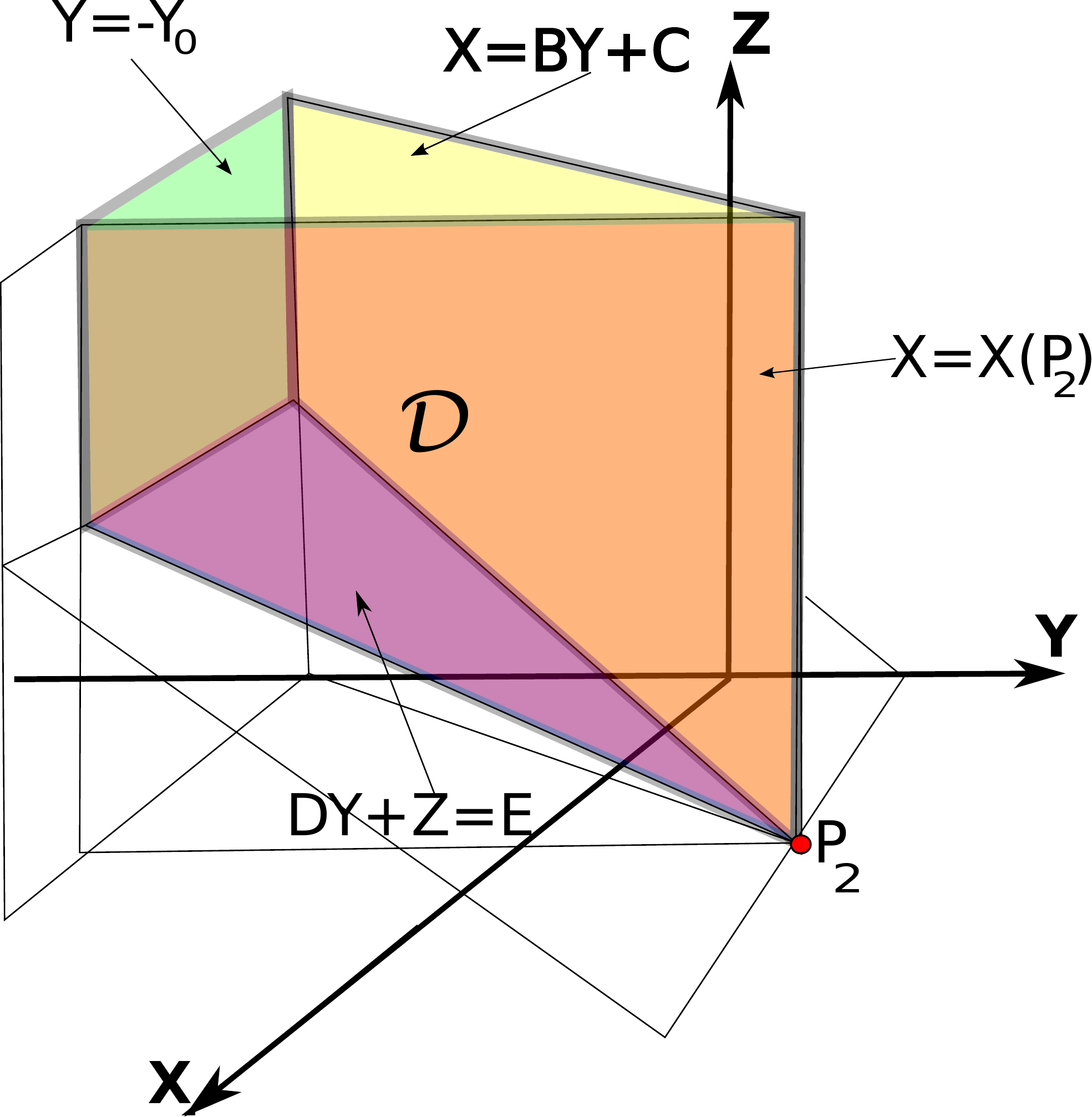}
  \end{center}
  \caption{The region $\mathcal{D}$ defined in \eqref{reg.large}.}\label{fig3}
\end{figure}

\medskip

\noindent \textbf{Step 3. Flow on the first plane}. The direction of the flow of the system \eqref{PSsyst1} over the plane \eqref{PlaneYZ} is given by the sign of the following rather complicated expression
\begin{equation}\label{flow.planeYZ}
\begin{split}
F(X,Y)&=-\frac{2(mN-N+2)^2p}{m-1}Y^2\\&-\frac{2(mN-N+2)[-(p-1)(m-1)\sigma+(m-1)(m-p)N-2p^2+2m+2p-2]}{(m-1)(\sigma+2)}Y\\
&+\frac{2(mN-N+2)^2[2(m-1)^2N^2+7(m-1)N+8-(m-1)\sigma]}{(m-1)^2}XY\\&-\frac{2(mN-N+2)L(m,p,N,\sigma)}{(m-1)^2(\sigma+2)}X\\
&=K(Y)X-\frac{2(mN-N+2)^2p}{m-1}Y^2\\&-\frac{2(mN-N+2)[-(p-1)(m-1)\sigma+(m-1)(m-p)N-2p^2+2m+2p-2]}{(m-1)(\sigma+2)}Y,
\end{split}
\end{equation}
where
\begin{equation*}
\begin{split}
L(m,p,N,\sigma)&=-(m-1)^2\sigma^2+[2(m-1)^3N^2+7(m-1)^2N-2(p-4)(m-1)]\sigma\\&+4(m-1)^2(p-1)N^2-2(m-1)(m-8p+7)N-4m+16p-12
\end{split}
\end{equation*}
and
\begin{equation*}
\begin{split}
K(Y)&=\frac{2(mN-N+2)^2[2(m-1)^2N^2+7(m-1)N+8-(m-1)\sigma]}{(m-1)^2}Y\\&-\frac{2(mN-N+2)L(m,p,N,\sigma)}{(m-1)^2(\sigma+2)}.
\end{split}
\end{equation*}
Let us notice first that, for $\sigma$ sufficiently large, the coefficient of $Y$ in the expression of $K(Y)$ is negative, while the free term is positive. It thus follows that $K(Y)>0$ for $Y\leq0$ and $\sigma$ large enough. Moreover, we are interested in the invariant region $\mathcal{D}$ defined in \eqref{reg.large}, thus $Y\leq Y(P_2)$ and we further have
$$
K(Y(P_2))=\frac{4(mN-N+2)(mN-pN+\sigma+2)}{(m-1)(\sigma+2)}>0,
$$
hence by linearity we infer that $K(Y)>0$ for any $Y\in(-\infty,Y(P_2))$. Therefore, the coefficient of $X$ in the expression of the flow \eqref{flow.planeYZ} is positive in the region of interest and this implies that, in the region $\mathcal{D}$, we get
$$
F(X,Y)>F(BY+C,Y)=\frac{(mN-N+2)^2(Y-Y(P_2))}{(\sigma+2)(m-1)[2(m-1)^2N^2+7(m-1)N+2(m+3)]}A(Y;\sigma),
$$
where
\begin{equation*}
A(Y;\sigma)=-[2m(m-1)Y+(m-1)^2(2mN-2N+3)]\sigma^2+K_1(Y)\sigma+K_2(Y),
\end{equation*}
with $K_1$ and $K_2$ continuous functions of $Y$ (thus bounded for $Y\in[-Y_0,Y(P_2)]$) whose expressions we omit here since for $\sigma$ sufficiently large the first term in $A(Y;\sigma)$ is the dominating one. Since
$$
-2m(m-1)Y_0+(m-1)^2(2mN-2N+3)=-(m-1)^2[(2N-1)(m-1)+2]<0,
$$
the coefficient of $\sigma^2$ is always negative for $Y\in[-Y_0,Y(P_2)]$ and thus $F(X,Y)>F(BY+C,Y)>0$ for $Y\in[-Y_0,Y(P_2)]$ and $X>BY+C$. This implies that orbits cannot leave the region $\mathcal{D}$ by crossing its wall given by the plane \eqref{PlaneYZ} as the flow across it points towards the interior of $\mathcal{D}$.

\medskip

\noindent \textbf{Step 4. Flow on the second plane}. We want in this step to prove a similar property as in Step 3 but for the wall of the region $\mathcal{D}$ given by the plane \eqref{PlaneXY}. The direction of the flow of the system \eqref{PSsyst1} over the plane \eqref{PlaneXY} is given by the sign of the expression
\begin{equation}\label{flow.planeXY}
H(Y,Z)=\frac{(m-1)(H_1(Y)Z+H_2Y^2+(\sigma+2)(mN-N+2)H_3Y+H_4)}{2(\sigma+2)^2[2(m-1)^2N^2+7(m-1)N+2(m+3)]^2},
\end{equation}
with
$$
H_1(Y)=2m^2(\sigma+2)^2(m-1)Y+m(m-1)(\sigma+2)(2mN-2N+3)[\sigma(m-1)+2(p-1)],
$$
$$
H_2=4m^2(\sigma+2)^2(mN-N+2)^2,
$$
and the lower order terms
\begin{equation*}
\begin{split}
H_3&=[4(m-1)^4N^2+2(m-1)^2(7m-6)N-(2m-3)^2(m-1)]\sigma+8(m-1)^2(p-1)N^2\\
&+4(m-1)^2(m+6p-6)N-8m^2p+14m^2+18mp-24m-18p+18
\end{split}
\end{equation*}
and finally
\begin{equation*}
\begin{split}
H_4&=(2mN-2N+3)(m-1)(\sigma(m-1)+2(p-1))\\
&\times[(2(m-1)^2N+4m-3)\sigma+4(p-1)(m-1)N+2m+6(p-1)].
\end{split}
\end{equation*}
We notice that $H_1(Y)$ is an affine function of $Y$ which is increasing, therefore
\begin{equation*}
\begin{split}
H_1(Y)>H_1(-Y_0)&=m(m-1)(\sigma+2)\left[(2(m-1)^2N-(m-1)(m-3))\sigma\right.\\&\left.+4(p-1)(m-1)N-2m^2+2m+6p-6\right]
\end{split}
\end{equation*}
which is positive for $\sigma$ sufficiently large. We infer that the coefficient of $Z$ in \eqref{flow.planeXY} is positive and thus $H(Y,Z)>H(Y,E-DY)$ in the region $\mathcal{D}$ introduced in \eqref{reg.large}. Since
\begin{equation*}
\begin{split}
H(Y,E-DY)&=\frac{(m-1)(mN-N+2)(2mN-2N+3)(Y(P_2)-Y)}{2(\sigma+2)[(2(m-1)^2N+4m-3)\sigma+4(p-1)(m-1)N+2m+6(p-1)]^2}\\
&\times\left[(2mN-2N+3)(m-1)\sigma+4(p-1)(m-1)N-2m^2+2mp+6(p-1)\right]>0
\end{split}
\end{equation*}
provided $\sigma>0$ is taken sufficiently large, we infer that an orbit entering the region $\mathcal{D}$ cannot leave it by crossing its wall defined by the plane \eqref{PlaneXY}.

\medskip

\noindent \textbf{Step 5. The orbit from $P_2$ enters the region $\mathcal{D}$}. We prove here that for $\sigma$ sufficiently large, the unique orbit going out of $P_2$ enters the region $\mathcal{D}$ introduced in \eqref{reg.large}. We know that the orbit goes out tangent to the eigenvector $e_3$ whose components $(X(\sigma),Y(\sigma),Z(\sigma))$ are given in \eqref{vector.P2}. We only need to compute the product between the direction of this vector and the normal directions to these planes and study its sign for $\sigma$ sufficiently large. For the plane \eqref{PlaneXY} we get
\begin{equation}\label{interm20}
\begin{split}
(1,-B,0)&\cdot(X(\sigma),Y(\sigma),Z(\sigma))\\&=\frac{m(m-1)^3}{(2(m-1)^2N+4m-3)\sigma+4(p-1)(m-1)N+2m+6(p-1)}\sigma\\
&+\frac{2m(m-1)^2(m+p-2)}{(2(m-1)^2N+4m-3)\sigma+4(p-1)(m-1)N+2m+6(p-1)}-(m-1)^3,
\end{split}
\end{equation}
while for the plane \eqref{PlaneYZ} we get
\begin{equation}\label{interm21}
\begin{split}
(0,D,1)&\cdot(X(\sigma),Y(\sigma),Z(\sigma))\\&=(m-1)^2\sigma^2-(m-1)[2(m-1)^2N^2+7(m-1)N-2(m+2p-5)]\sigma\\
&-4(m-1)^2(m+p-2)N^2-4(m-1)(3m+4p-7)N\\&+4mp+4p^2-12m-24p+28,
\end{split}
\end{equation}
and both scalar products in \eqref{interm20} and \eqref{interm21} are positive for $\sigma$ large, ending this step.

\medskip

\noindent \textbf{Step 6. No $\omega$-limits. The orbit from $P_2$ enters $Q_3$}. We established in Step 5 that the orbit going out of $P_2$ for $\sigma$ sufficiently large must enter the region $\mathcal{D}$ and once there, it has two alternatives: either go out of $\mathcal{D}$ in some way, or remain there and go to a stable $\omega$-limit, if any. We prove here that the latter is not possible. We infer from Lemma \ref{lem.monot} that coordinates $X$ and $Y$ are monotone along the orbit going out of $P_2$ as long as $Y>0$, thus no $\omega$-limit different from critical points can exist in the region $\mathcal{D}\cap\{Y>0\}$. Assume for contradiction that the orbit from $P_2$ ends up in a $\omega$-limit orbit lying in the region $\mathcal{D}\cap\{-Y_0<Y\leq0\}$ for some $\sigma_0>0$ sufficiently large (such that the previous steps of the proof are fulfilled). This $\omega$-limit is invariant for the flow according to \cite[Theorem 2, Section 3.2]{Pe}. Since $X$ is monotone decreasing along all the orbit going out of $P_2$, it should have a finite limit, thus

$\bullet$ either $X(\eta)\to\delta>0$ as $\eta\to\infty$ along the orbit from $P_2$, which means that the periodic $\omega$-limit is included in the plane $\{X=\delta\}$ and by invariance, it is itself a solution to the system (hence, to \eqref{SSODE} if translated into profiles). But $X=\delta$ implies
$$
f(\xi)=\delta\xi^{2/(m-1)}, \qquad {\rm for \ any} \ \xi>0,
$$
and it is easy to check by direct calculation that there are no solutions to \eqref{SSODE} having this form, reaching a contradiction.

$\bullet$ or $X(\eta)\to0$ as $\eta\to\infty$, which implies the existence of a periodic orbit lying in the plane $\{X=0\}$ and the region $\{-Y_0<Y<0\}$. But this is easy to be ruled out, for example by noticing that on such orbit
$$
\dot{Z}=(p-1)YZ\leq0,
$$
hence the component $Z$ is either decreasing or constant, leading again to a contradiction.

We thus find that the orbit going out of $P_2$ has to quit the region $\mathcal{D}$ for $\sigma$ sufficiently large, and this is mandatory to be done by crossing the plane $\{Y=-Y_0\}$, as shown in the previous steps. Then the orbit has to enter the attractor $Q_3$, as proved in Step 1.

\medskip

\noindent \textbf{Step 7. No orbits from $Q_1$ to $P_1$. Good orbits only from $P_0$}. In order to complete the proof of Part 3 of Theorem \ref{th.class}, we are left with showing that for $\sigma$ sufficiently large, no orbit going out of $Q_1$ might enter $P_1$. We in fact show that all the orbits going out of $Q_1$ end up in the region $\mathcal{D}$ introduced in \eqref{reg.large}. To this end, we recall from the Remark after Lemma \ref{lem.Q1} that the orbits going out of $Q_1$ start with $Y=1/N>Y(P_2)$ by \eqref{ineq1}. This shows that, on the one hand, on such orbits $Z>E-DY$ for any $Y\geq Y(P_2)$, since $E-DY\leq0$ if $Y\geq Y(P_2)$. On the other hand, we go back to the system \eqref{systinf1} centered in $Q_1$ and consider in variables $(y,z,w)$ the cylinder of equation
\begin{equation}\label{cyl}
-(N-2)y+w-my^2-\frac{\beta}{\alpha}yw=0,
\end{equation}
which connects $Q_1$ to $P_2$. The idea of this cylinder is given by the isocline of the reduced system \eqref{systinf1} in the invariant plane $\{z=0\}$. The direction of the flow of the system \eqref{systinf1} across this cylinder is given by the sign of
$$
T(y,z,w)=z\left[N-2+2my+\frac{\beta}{\alpha}w\right]+\left(1-\frac{\beta}{\alpha}y\right)w(2-(m-1)y)>0,
$$
provided $0\leq y\leq y(P_2)$ where $y(P_2)=2/(m-1)$ is the first coordinate of the critical point $P_2$ in variables $(y,z,w)$. Since the direction of the normal vector to the cylinder \eqref{cyl} is given by
$$
\overline{n}=\left(-(N-2)-2my-\frac{\beta}{\alpha}w,0,1-\frac{\beta}{\alpha}y\right),
$$
it follows that this cylinder cannot be crossed from inside to its exterior while $0\leq y\leq y(P_2)=2/(m-1)$. Moreover, by analyzing the matrix $M(Q_1)$ given in the proof of Lemma \ref{lem.Q1}, it follows that the orbits going out of $Q_1$ go out (in a small neighborhood of $Q_1$) with a slope given by
$$
\frac{w}{y}=\frac{1}{Y}\sim N,
$$
while the slope of the cylinder \eqref{cyl} in a neighborhood of $Q_1$ (where the linear terms dominate over the quadratic ones) is given by $w/y\sim(N-2)<N$. We thus infer that all the orbits going out of $Q_1$ are interior to the cylinder \eqref{cyl} in a small neighborhood of $Q_1$ and thus will remain in the same situation while $y\leq y(P_2)$. It is then easy to see from the expression of \eqref{cyl} and the form $X=BY+C$ of the plane \eqref{PlaneXY}, which translates into $1=By+Cw$ in terms of the variables of the system \eqref{systinf1}, that the cylinder \eqref{cyl} and thus also the orbits going out of $Q_1$ stay in the region $1>By+Cw$ (that is, $X>BY+C$) up to arriving at $Y=Y(P_2)$. We have thus shown that all the orbits going out of $Q_1$ enter the region $\mathcal{D}$ and thus for $\sigma$ sufficiently large they cannot reach $P_1$. Theorem \ref{th.exist} then gives that the only good profiles for $\sigma$ sufficiently large have to belong to orbits coming from $P_0$, as claimed.
\end{proof}
We are left with the second statement in Theorem \ref{th.class}, which is proved below.
\begin{proof}[Proof of Theorem \ref{th.class}, Part 2]
We want to show that there exists at least a value $\sigma^*>0$ such that the orbit going out of $P_2$ enters $P_1$ in the phase space associated to the system \eqref{PSsyst1}. To this end, let $\sigma>0$. We infer from Lemma \ref{lem.monot} that the $X$ coordinate is decreasing along the orbit going out of $P_2$, and the same happens for the $Y$ coordinate while $Y\geq0$. This, together with an argument completely similar to \textbf{Step 6} in the above proof, gives that the orbit going out of $P_2$ cannot enter a limit cycle: indeed, due to the monotonicity in $X$, this limit cycle should lie in the invariant plane $\{X=0\}$. Since the line $\{Y=0\}$ is invariant in the plane $\{X=0\}$ and $\dot{Z}=(p-1)YZ$, it follows that $Z$ is monotone in both half-planes $\{Y>0\}$ and $\{Y<0\}$, thus no limit cycles may exist inside $\{X=0\}$. We then infer that the orbits going out of $P_2$ must reach a critical point among $P_{\gamma}$, $P_1$ and $Q_3$. We thus define the three sets, with respect to $\sigma$
\begin{equation*}
\begin{split}
&\mathbb{A}=\{\sigma\in(0,\infty):{\rm the \ orbit \ from} \ P_2 \ {\rm enters} \ P_{\gamma}\},\\
&\mathbb{B}=\{\sigma\in(0,\infty):{\rm the \ orbit \ from} \ P_2 \ {\rm enters} \ P_{1}\}, \\
&\mathbb{C}=\{\sigma\in(0,\infty):{\rm the \ orbit \ from} \ P_2 \ {\rm enters} \ Q_3\}.
\end{split}
\end{equation*}
Since $P_{\gamma}$ is asymptotically stable for orbits coming from the positive part of the phase space and $Q_3$ is a stable node, sets $\mathbb{A}$ and $\mathbb{C}$ are open. We furthermore infer from \textbf{Step 6} and the above considerations on the orbits from $P_2$ that these two sets are also non-empty. It follows by basic topology that the set $\mathbb{B}$ is non-empty and closed, thus it contains at least an element $\sigma^*$, as claimed.
\end{proof}

\bigskip

\noindent \textbf{Acknowledgements} R. I. and A. S. are partially supported by the Spanish project PID2020-115273GB-I00.

\bibliographystyle{plain}

\begin{thebibliography}{1}

\bibitem{AdB91}
D. Andreucci and E. DiBenedetto, \emph{On the Cauchy problem and initial traces for a class of evolution equations with strongly nonlinear sources}, Ann. Scuola Norm. Sup. Pisa, \textbf{18} (1991).

\bibitem{AT05}
D. Andreucci and A. F. Tedeev, \emph{Universal bounds at the
blow-up time for nonlinear parabolic equations}, Adv. Differential
Equations, \textbf{10} (2005), no. 1, 89-120.


\bibitem{BZZ11}
X. Bai, S. Zhou and S. Zheng, \emph{Cauchy problem for fast
diffusion equation with localized reaction}, Nonlinear Anal.,
\textbf{74} (2011), no. 7, 2508-2514.

\bibitem{BL89}
C. Bandle and H. Levine, \emph{On the existence and nonexistence of
global solutions of reaction-diffusion equations in sectorial
domains}, Trans. Amer. Math. Soc., \textbf{316} (1989), 595-622.

\bibitem{BG84}
P. Baras and J. Goldstein, \emph{The heat equation with a singular potential}, Trans. Amer. Math. Soc., \textbf{284} (1984), no. 1, 121-139.

\bibitem{BK87}
P. Baras and R. Kersner, \emph{Local and global solvability of a
class of semilinear parabolic equations}, J. Differential Equations,
\textbf{68} (1987), 238-252.

\bibitem{BS19}
B. Ben Slimene, \emph{Asymptotically self-similar global solutions for Hardy-H\'enon parabolic systems}, Differ. Equ. Appl., \textbf{11} (2019), no. 4, 439-462.

\bibitem{BSTW17}
B. Ben Slimene, S. Tayachi and F. B. Weissler, \emph{Well-posedness, global existence and large time behavior for Hardy-H\'enon parabolic equations}, Nonlinear Anal., \textbf{152} (2017), 116-148.

\bibitem{CM99}
X. Cabr\'e and Y. Martel, \emph{Existence versus explosion instantan\'ee por des \'equations de la chaleur lin\'eaires avec potentiel singulier}, C. R. Acad. Sci. Paris, \textbf{329} (1999), no. 11, 973-978.

\bibitem{Carr}
J. Carr, \emph{Applications of Centre Manifold Theory}, Springer Verlag, New York, 1981.

\bibitem{CIT21a}
N. Chikami, M. Ikeda and K. Taniguchi, \emph{Well-posedness and global dynamics for the critical Hardy-Sobolev parabolic equation}, Nonlinearity, \textbf{34} (2021), no. 11, 8094-8142.

\bibitem{CIT21b}
N. Chikami, M. Ikeda and K. Taniguchi, \emph{Optimal well-posedness and forward self-similar solution for the Hardy-H\'enon parabolic equation in critical weighted Lebesgue spaces}, Nonlinear Anal., to appear, Preprint ArXiv no. 2104.14166.

\bibitem{CdPE98}
C. Cort\'azar, M. del Pino and M. Elgueta, \emph{On the blow-up set for $u_t=\Delta u^m+u^m$, $m>1$}, Indiana Univ. Math. J., \textbf{47} (1998), 541--562.

\bibitem{CdPE02}
C. Cort\'azar, M. del Pino and M. Elgueta, \emph{Uniqueness and stability of regional blow-up in a porous-medium equation}, Ann. Inst. H. Poincar\'e Analyse Non Lin\'eaire, \textbf{19} (2002), no. 6, 927--960.

\bibitem{FdP18}
R. Ferreira and A. de Pablo, \emph{Grow-up for a quasilinear heat equation with a localized reaction in higher dimensions}, Rev. Mat. Complut., \textbf{31} (2018), no. 3, 805-832.

\bibitem{FdPV06}
R. Ferreira, A. de Pablo and J. L. V\'azquez, \emph{Classification
of blow-up with nonlinear diffusion and localized reaction}, J.
Differential Equations, \textbf{231} (2006), no. 1, 195-211.

\bibitem{FT00}
S. Filippas and A. Tertikas, \emph{On similarity solutions of a heat equation with a nonhomogeneous nonlinearity}, J. Differential Equations, \textbf{165} (2000), no. 2, 468-492.


\bibitem{GU05}
Y. Giga and N. Umeda, \emph{Blow-up directions at space infinity for solutions of semilinear heat equations}, Bol. Soc. Paran. Mat., \textbf{23} (2005), 9-28.

\bibitem{GU06}
Y. Giga and N. Umeda, \emph{On blow-up at space infinity for semilinear heat equations}, J. Math. Anal. Appl., \textbf{316} (2006), 538-555.

\bibitem{GK03}
J. A. Goldstein and I. Kombe, \emph{Nonlinear degenerate prabolic equations with singular lower-order term}, Adv. Differential Equations, \textbf{8} (2003), no. 10, 1153-1192.

\bibitem{GGK05}
G. R. Goldstein, J. A. Goldstein, and I. Kombe, \emph{Nonlinear parabolic equations with singular coefficient and critical exponent}, Appl. Anal., \textbf{84} (2005), no. 6, 571-583.

\bibitem{GH}
J. Guckenheimer and Ph. Holmes, \emph{Nonlinear oscillation, dynamical systems and bifurcations of vector fields}, Applied Mathematical Sciences, vol. 42, Springer-Verlag, New York, 1990.

\bibitem{GLS}
J.-S. Guo, C.-S. Lin and M. Shimojo, \emph{Blow-up behavior for a
parabolic equation with spatially dependent coefficient}, Dynam.
Systems Appl., \textbf{19} (2010), no. 3-4, 415-433.

\bibitem{GS11}
J.-S. Guo and M. Shimojo, \emph{Blowing up at zero points of
potential for an initial boundary value problem}, Commun. Pure Appl.
Anal., \textbf{10} (2011), no. 1, 161-177.

\bibitem{GLS13}
J.-S. Guo, C.-S. Lin and M. Shimojo, \emph{Blow-up for a
reaction-diffusion equation with variable coefficient}, Appl. Math.
Lett., \textbf{26} (2013), no. 1, 150-153.

\bibitem{GS18}
J.-S. Guo, and P. Souplet, \emph{Excluding blowup at zero points of the potential by means of Liouville-type theorems},
J. Differential Equations, \textbf{265} (2018), no. 10, 4942-4964.

\bibitem{HT21}
K. Hisa and J. Takahashi, \emph{Optimal singularities of initial data for solvability of the Hardy parabolic equation}, J. Differential Equations, \textbf{296} (2021), 822-848.

\bibitem{IL13}
R. G. Iagar and Ph.~Lauren\ced{c}ot, \emph{Existence and uniqueness
of very singular solutions for a fast diffusion equation with
gradient absorption},  J. London Math. Soc., \textbf{87} (2013),
509-529.

\bibitem{ILS22}
R. G. Iagar, Ph. Lauren\ced{c}ot and A. S\'anchez, \emph{Self-similar shrinking of supports and non-extinction for a nonlinear diffusion equation with strong nonhomogeneous absorption}, Submitted (2022), Preprint ArXiv no. 2204.09307.
%

\bibitem{IMS21}
R. G. Iagar, A. I. Mu\~{n}oz and A. S\'anchez, \emph{Self-similar solutions preventing finite time blow-up for reaction-diffusion equations with singular potential}, Submitted (2021), Preprint ArXiv no. 2111.04806.

\bibitem{IMS22}
R. G. Iagar, A. I. Mu\~{n}oz and A. S\'anchez, \emph{Self-similar blow-up patterns for a reaction-diffusion equation with weighted reaction in general dimension}, Comm. Pure Appl. Analysis, \textbf{21} (2022), no. 3, 891-925.

\bibitem{IS19}
R. G. Iagar and A. S\'anchez, \emph{Blow up profiles for a quasilinear reaction-diffusion equation with weighted reaction with linear growth}, J. Dynam. Differential Equations, \textbf{31} (2019), no. 4, 2061-2094.

\bibitem{IS20}
R. G. Iagar and A. S\'anchez, \emph{Instantaneous and finite time blow-up of solutions to a reaction-diffusion equation with Hardy-type singular potential}, J. Math. Anal. Appl., \textbf{491} (2020), no. 1, paper no. 124244, 11 pages.

\bibitem{IS21a}
R. G. Iagar and A. S\'anchez, \emph{Blow up profiles for a quasilinear reaction-diffusion equation with weighted reaction}, J. Differential Equations, \textbf{272} (2021), no. 1, 560-605.

\bibitem{IS22}
R. G. Iagar and A. S\'anchez, \emph{Separate variable blow-up patterns for a reaction-diffusion equation with critical weighted reaction}, Nonlinear Anal., \textbf{217} (2022), article no. 112740, 33 pages.



\bibitem{KWZ11}
X. Kang, W. Wang and X. Zhou, \emph{Classification of solutions of
porous medium equation with localized reaction in higher space
dimensions}, Differential Integral Equations, \textbf{24} (2011),
no. 9-10, 909-922.

\bibitem{Ko04}
I. Kombe, \emph{Doubly nonlinear parabolic equations with singular lower order term}, Nonlinear Anal., \textbf{56} (2004), no. 2, 185-199.

\bibitem{La84}
A. A. Lacey, \emph{The form of blow-up for nonlinear parabolic equations}, Proc. Royal Society Edinburgh Sect. A, \textbf{98} (1984), no. 1-2, 183-202.

\bibitem{M78}
H. Matano, \emph{Convergence of solutions of one-dimensional semilinear parabolic equations}, J. Math. Kyoto Univ., \textbf{18} (1978), no. 2, 221-227.

\bibitem{MS21}
A. Mukai and Y. Seki, \emph{Refined construction of Type II blow-up solutions for semilinear heat equations with Joseph-Lundgren supercritical nonlinearity}, Discrete Cont. Dynamical Systems, \textbf{41} (2021), no. 10, 4847-4885.

\bibitem{dPS00}
A. de Pablo and A. S\'anchez, \emph{Global travelling waves in reaction-convection-diffusion equations}, J. Differential Equations, \textbf{165} (2000), no. 2, 377-413.

\bibitem{Pe}
L. Perko, \emph{Differential equations and dynamical systems. Third
edition}, Texts in Applied Mathematics, \textbf{7}, Springer Verlag,
New York, 2001.

\bibitem{Pi97}
R. G. Pinsky, \emph{Existence and nonexistence of global solutions
for $u_t=\Delta u+a(x)u^p$ in $\real^d$}, J. Differential Equations,
\textbf{133} (1997), no. 1, 152-177.

\bibitem{Pi98}
R. G. Pinsky, \emph{The behavior of the life span for solutions to
$u_t=\Delta u+a(x)u^p$ in $\real^d$}, J. Differential Equations,
\textbf{147} (1998), no. 1, 30-57.

\bibitem{Qi98}
Y.-W. Qi, \emph{The critical exponents of parabolic equations and blow-up in $\real^n$}, Proc. Royal Soc. Edinburgh A, \textbf{128} (1998), 123-136.

\bibitem{QS}
P. Quittner and Ph. Souplet, \emph{Superlinear parabolic problems.
Blow-up, global existence and steady states}, Birkhauser Advanced
Texts, Birkhauser Verlag, Basel, 2007.

\bibitem{S4}
A. A. Samarskii, V. A. Galaktionov, S. P. Kurdyumov and A. P.
Mikhailov, \emph{Blow-up in quasilinear parabolic problems}, de
Gruyter Expositions in Mathematics, \textbf{19}, W. de Gruyter,
Berlin, 1995.

\bibitem{S73}
J. Sotomayor, \emph{Generic bifurcations of dynamical systems}, in Proceedings of a Symposium Held at University of Bahia, Salvador, Brasil, Academic Press, New York, 1973, 561-582.

\bibitem{Su02}
R. Suzuki, \emph{Existence and nonexistence of global solutions of
quasilinear parabolic equations}, J. Math. Soc. Japan, \textbf{54}
(2002), no. 4, 747-792.

\bibitem{T20}
S. Tayachi, \emph{Uniqueness and non-uniqueness of solutions for critical Hardy-H\'enon parabolic equations}, J. Math. Anal. Appl. \textbf{488} (2020), no. 1, paper no. 123976, 51 pages.

\end{thebibliography}

\end{document}